\documentclass[12pt]{amsart}
\usepackage{graphicx}

\usepackage{amsmath,graphics,enumerate}
\usepackage{amsfonts,amssymb}
\usepackage[all]{xypic}
\usepackage{txfonts}

\theoremstyle{plain}
\newtheorem*{theorem*}{Theorem}
\newtheorem*{lemma*} {Lemma}
\newtheorem*{corollary*} {Corollary}
\newtheorem*{proposition*} {Proposition}
\newtheorem*{conjecture*}{Conjecture}

\newtheorem{theorem}{Theorem}[section]
\newtheorem{lemma}[theorem]{Lemma}
\newtheorem{corollary}[theorem]{Corollary}
\newtheorem{proposition}[theorem]{Proposition}
\newtheorem{conjecture}[theorem]{Conjecture}

{\catcode`@=11\global\let\c@figure=\c@theorem}

{\catcode`@=11\global\let\c@equation=\c@theorem}

\theoremstyle{remark}
\newtheorem*{remark}{Remark}

\theoremstyle{definition}

\numberwithin{equation}{section}
\numberwithin{figure}{section}

\def\be{\begin{equation}}
\def\ee{\end{equation}}

\def\bing{\mbox{Bing}}

\def\Z{\Bbb{Z}}

\def\hJ{\hat{J}}

\def\E{\Bbb{E}}
\def\part{\partial}

\def\dtimes{D^2\hspace{-0.07cm}\times\hspace{-0.07cm} [0,1]}
\def\bp{\begin{pmatrix}}
\def\ep{\end{pmatrix}}
\def\bn{\begin{enumerate}}

\def\en{\end{enumerate}}
\def\ba{\begin{array}}
\def\ea{\end{array}}

\def\fr12{\frac{1}{2}}

\def\int{\mbox{Int}}
\def\dcup{\,\cup \,}

\def\into{\hookrightarrow}

\def\v{\varphi}

\def\longisoarrow{\xrightarrow{\,\,\cong\,\,}}

\def\sm{\smallsetminus}

\begin{document}

\title{New constructions of slice links}
\author{Tim Cochran, Stefan Friedl and Peter Teichner}

\address{Rice University, Houston, TX 77005}
\email{cochran@rice.edu}
\address{Universit\'e du Qu\'ebec \`a Montr\'eal, Montr\'eal, Qu\`ebec}
\email{friedl@alumni.brandeis.edu}
\address{University of California, Berkeley, CA 94720}
\email{teichner@math.berkeley.edu}
\def\subjclassname{\textup{2000} Mathematics Subject Classification}
\expandafter\let\csname subjclassname@1991\endcsname=\subjclassname
\expandafter\let\csname subjclassname@2000\endcsname=\subjclassname

\date{\today}
\begin{abstract}
We use techniques of Freedman and Teichner \cite{FnT95} to prove that under certain circumstances the multi-infection of a slice link is again slice (not necessarily smoothly slice). We provide a general context for proving links are slice that includes many of the previously known results.
\end{abstract}
\maketitle

\section{Introduction}\label{sec:intro}
A {\em link of $m$ components} is the image of a flat embedding $S^1\amalg\dots\amalg S^1 \into S^{3}$ of the ordered disjoint union of $m$ oriented copies of the circle $S^1$. Two such links are called {\em concordant} if there exists a flat embedding 
\[
(S^{1}\amalg \dots \amalg S^{1}) \times [0,1] \into S^{3}\times [0,1]
\]
 which restricts to the given links at the ends.  A link is called (topologically) {\em slice} if it is concordant to
the trivial $m$--component link or, equivalently, if it bounds a flat embedding of $m$ disjoint {\em slice} disks $D^{2}\amalg\dots\amalg D^2 \into D^{4}$. In the special case $m=1$ we refer to the link as a knot. If the embeddings above are required to be $C^\infty$, or {\em smooth}, then these notions are called \emph{smoothly concordant} and \emph{smoothly slice}.

The study of link concordance was initiated by Fox and Milnor in the early $1960's$ arising from their study of isolated singularities of $2$-spheres in $4$-manifolds. It is now known that specific questions about link concordance are equivalent to whether or not the surgery and s-cobordism theorems (that hold true in higher dimensions) hold true for topological $4$-manifolds. Moreover, the difference between a link being topologically slice and being smoothly slice can be viewed as ``atomic'' for the existence of multiple differential structures on a fixed topological $4$-manifold.

There is only one known way to construct a smoothly slice link, namely as the boundary of a set of \emph{ribbon disks} ~\cite{Ro90}. The known constructions of (topologically) slice links are also fairly limited. In $1982$ Michael Freedman proved that any knot with Alexander polynomial $1$ is slice ~\cite{Fr85}. It is known that some of these knots cannot be smoothly slice and hence cannot arise from the ribbon construction. Freedman ~\cite{Fr85,Fr88} and later Freedman and Teichner ~\cite{FnT95} gave other techniques showing that the \emph{Whitehead doubles} of various links are slice.
The $4$-dimensional surgery and s-cobordism theorems (for all fundamental groups) are in fact equivalent to the {\em free} sliceness of Whitehead doubles of all links with vanishing linking numbers, see \cite{FQ90}. Here a link is {\em freely} slice if the complement of some set of slice disks in $D^4$ has free fundamental group. However, it is conjectured that:

\begin{conjecture}\label{conj}
The Whitehead double of a link is freely slice if and only if the link is homotopically trivial (i.e.\ has vanishing non-repeating Milnor $\bar\mu$-invariants).
\end{conjecture}
\noindent Since vanishing linking numbers corresponds to vanishing Milnor invariants of length~2, the above conjecture (applied to, say, the Borromean rings) would imply that one of those theorems does not hold for free groups. The conjecture is known for links with one or two components (\cite{Fr88}) but is widely open for all other cases, the harder part being the ``only if'' direction. 
Continuing the history of constructions for slice links, the second two authors recently found a new technique for knots, including examples that are not ribbon knots, do not have Alexander polynomial $1$ and are not Whitehead doubles ~\cite{FlT05,FlT06}.

In the present paper we discuss a method of constructing slice links that generalizes many of the above. The construction begins with a ribbon knot or link and modifies it by a procedure called a \emph{multi-infection}
 (previously called \emph{infection by a string link} ~\cite[p.~385]{C04} and a \emph{tangle sum} ~\cite[Section~1]{CO94}) which generalizes the classical satellite construction. Special cases of this
 construction have been used extensively since the late $1970$'s to exhibit interesting examples of knots and links
  that are \emph{not slice} ~\cite{Gi83,Li05,COT03,COT04,Ha06,Ci06}. Therefore the present paper complements these
  results, giving hope for an eventual complete resolution of the question of when this construction results in a slice knot or link.
Our result also provides a method of producing interesting examples for testing the new
obstructions to a knot or link being smoothly slice \cite{OS03,Ra04,MO05,BW05,GRS07}.

In order to state our main theorem we now define the multi-infection of a link by a string link. By an \emph{$r$--multi--disk} $\E$
we mean the oriented disk $D^2$ together with $r$ ordered embedded open disks $E_1,\dots,E_r$
\begin{figure}[h] \centering
  \includegraphics[scale=0.3]{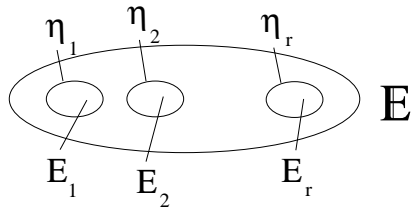}
\caption{Multi--disk.} \label{multi_disk}
\end{figure}
(cf. Figure \ref{multi_disk}). Given a link $L\subset S^3$ we say that a map $\varphi:\E\to S^3$
of an $r$--multi--disk into $S^3$ is {\em proper} if it is an embedding such that the image of the
multi--disk (which we denote by $\E_\varphi$) intersects the link components transversely and only
in the images of the disks $E_1,\dots,E_r$
\begin{figure}[h] \centering
  \includegraphics[scale=0.3]{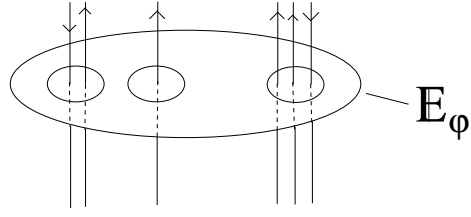}
\caption{Properly embedded multi--disk.} \label{multi_infection}
\end{figure}
as in Figure \ref{multi_infection}. Now let $J=J_1,\dots,J_r\subset D^2\times [0,1]$ be an (unoriented) $r$-component string link. Then we can
thicken up $\E_\varphi\subset S^3$ using the orientation of $\E_\varphi$, and tie $J$ into $L$
along $\E_\varphi$ (cf. Figure \ref{fig:infection}).
\begin{figure}[h] \centering
\begin{tabular}{cc}
  \includegraphics[scale=0.3]{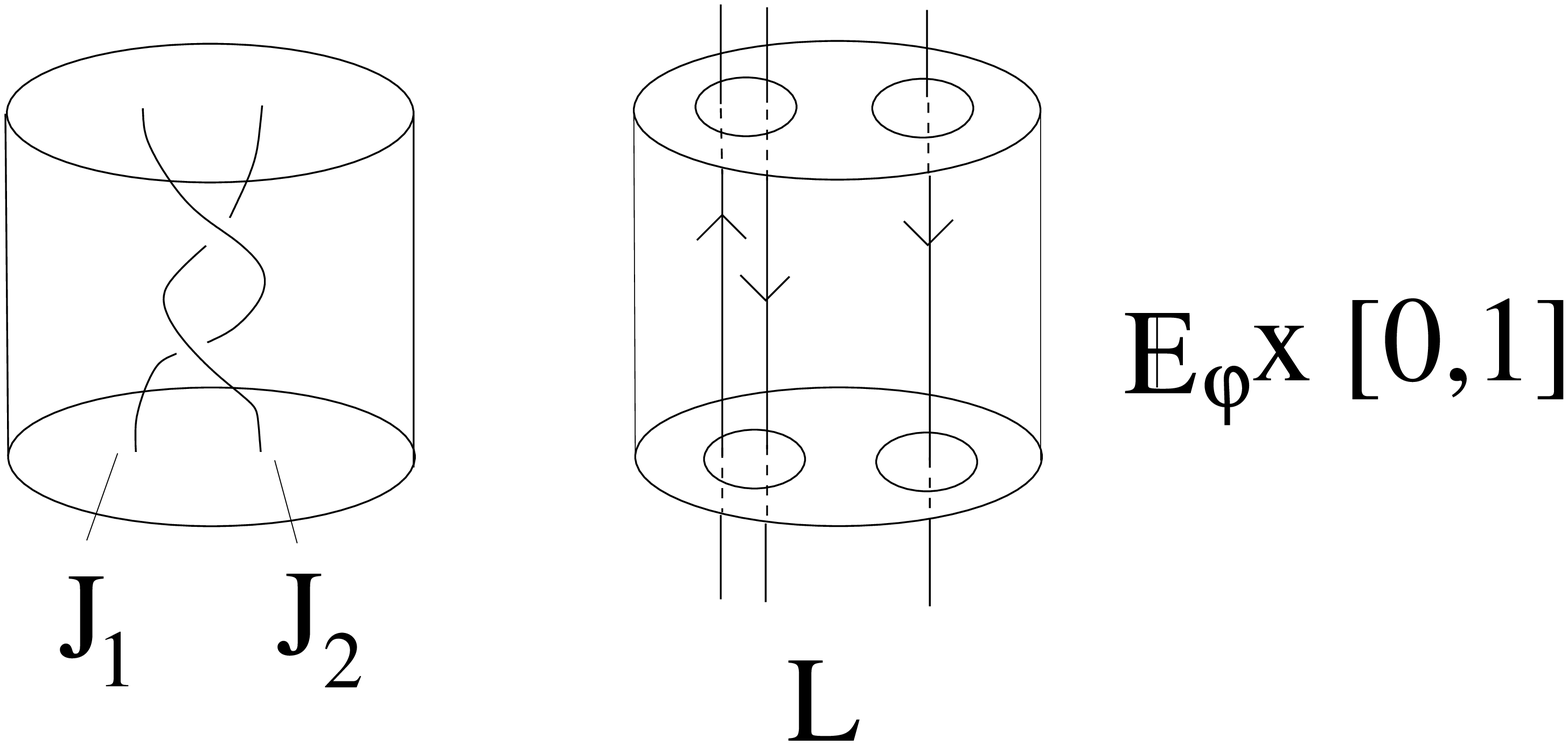}&
\includegraphics[scale=0.3]{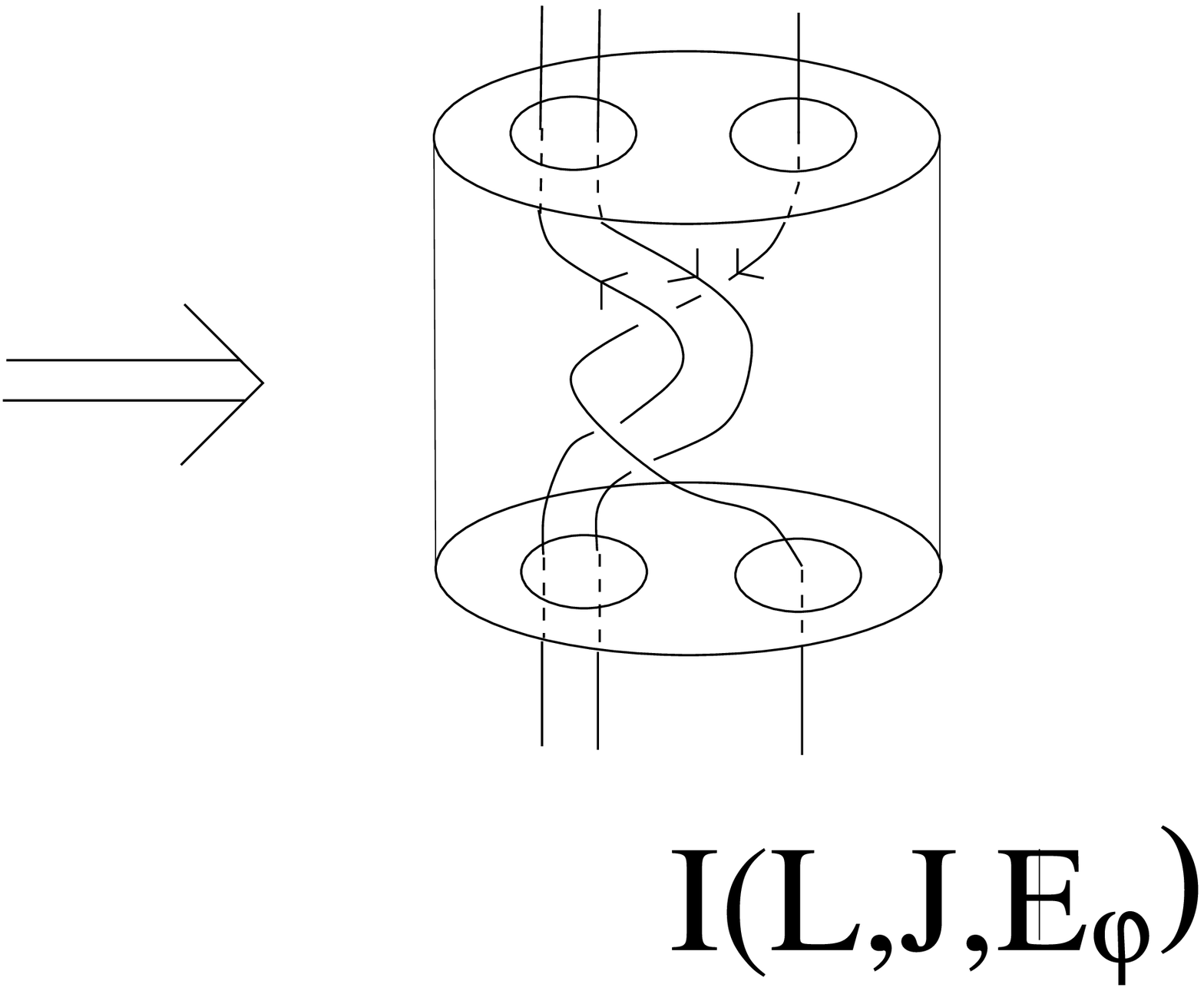}
\end{tabular}
\caption{Multi-infection of $L$ by $J$.} \label{fig:infection}
\end{figure}
We call the resulting link the multi--infection of $L$ by $J$ along $\E_\varphi$ and denote it by
$I(L,J,\E_\varphi)$. We refer to Section \ref{section:bound} for a more formal definition. We will always refer to the image of the boundary curves of
$\varphi(E_1),\dots,\varphi(E_r)$ by $\eta_1,\dots,\eta_r$. Note that in the case $r=1$, the multi--infection of a link $L$ by a string knot $J$ along a 1--multi--disk $\E_\v$ depends only on the curve $\eta$ and on the closure $\hat{J}$. In fact the resulting link is just the satellite link of $L$ with companion $\hJ$ and axis $\eta$, that we denote $I(L,\hat{J},\eta)$.

We can now state our main theorem which is an application of the techniques of \cite{FnT95} (see also \cite{Kr03}). We refer to \cite{Mi57} for the definition of the $\bar{\mu}$--invariants for links. For string links, these can actually be defined without indeterminacy but we will not need that fact here since the vanishing of the $\bar\mu$-invariants up to a certain length is well defined and depends only on the link closure of the string link, compare Figure~\ref{stringlink}.

\begin{theorem} \label{mainthm}
Let $D=D_1\amalg \dots\amalg D_m \hookrightarrow D^4$ be slice disks for a link $L$ in $S^3$.
Let $\varphi:\E\to S^3$ be a proper map of an $r$--multi--disk such that $\eta_1,\dots,\eta_r$ bound a set of immersed disks $\delta_i$ in $D^4\sm D$ in general position. Let $c$ be the total number of intersection and
self--intersection points of the $\delta_i$ and let $J$ be an $r$--component string link with vanishing Milnor $\bar{\mu}$--invariants up to (and including) length $2c$.

Then the multi--infection $I(L,J,\E_\varphi)$ of $L$ by $J$ along $\E_\varphi$  is also slice.
\end{theorem}

It is not hard to show that the theorem holds for $c=0$. We will therefore assume below that $c>0$ and in particular that the string link $J$ has trivial linking numbers (i.e. $\bar\mu$-invariants of length~2).

Note that, in the case $r=1$, $\hat{J}$ is a knot, and hence all Milnor's $\bar{\mu}$--invariants of $\hat{J}$ are zero.
In this case Theorem \ref{mainthm} simplifies to the following.

\begin{corollary} \label{corthm}
Let $D=D_1\amalg \dots\amalg D_m \hookrightarrow D^4$ be slice disks for a link $L$ in $S^3$.
Let $\eta$ be a closed curve in $S^3\sm L$, unknotted in $S^3$, such that $\eta$ is trivial in $\pi_1(D^4\sm D)$.
Then for any knot $\hat{J}$ the satellite link $I(L,\hat{J},\eta)$ is slice.
\end{corollary}

We note that our proof of Theorem \ref{mainthm} would go through under the weaker assumption that $J$ has trivial linking numbers if the 4-dimensional surgery sequence were exact for all fundamental groups. The latter is still an open problem. We use the assumption on the vanishing of higher Milnor-invariants of $J$ to get ourselves into the $\pi_1$-null setting where Freedman and Quinn \cite{FQ90} proved a surgery theorem up to s-cobordism. Conversely, if our theorem were true under this weaker assumption on $J$ then the surgery sequence would be exact and the s-cobordism theorem would hold for arbitrary fundamental groups. This follows from the following discussion and the comments below Conjecture~\ref{conj}.

\begin{remark}
An important special case of the theorem is when $L$ is the trivial link and $D$ is a set of slice disks coming from disks in 3-space. Take $r=m$ and choose $\eta_i$ in such a way that $(L_i,\eta_i)$ form Whitehead links in disjoint 3-balls. Then there are obvious immersed disks $\delta_i$ bounding $\eta_i$ each having exactly one self-intersection and no other intersections. This means that $c=m$ in the above theorem. Using the symmetry of the Whitehead link, one can redraw the picture so that the $\eta_i$ lie in a plane that also contains a multi--disk $\E$. It is then not hard to see that the multi--infection $I(L,J,\E_\varphi)$ is the Whitehead double of the closure of $J$. Our theorem thus implies that this Whitehead double is (freely) slice if the $\bar\mu$-invariants of $J$ vanish up to length $2m$. Theorem 3.1 in \cite{FnT95} gives the same result with the better bound $m+1$ rather than $2m$. This is the best known result concerning the ``if'' part of Conjecture~\ref{conj} above (the Conjecture implies that $m+1$ can be improved to $m$). Note that our current theorem vastly generalizes this very special case and hence it is not surprising that we need a slightly stronger assumption on the link $J$ in the general setting.
\end{remark}

Theorem~\ref{mainthm} places conditions on both the link $J$ (having vanishing Milnor invariants up to a certain length) and the curves $\{\eta_i\}$ (being null-homotopic in $D^4\sm D$). In general both these conditions are necessary. For example, in case $L$ and $J$ are knots ($m=r=1$), if the condition on $\eta$ is relaxed then in many cases $I(L,J,\E_\varphi)$ is provably \emph{not slice}, despite the fact that all the Milnor invariants vanish for $J$. In the case that $L$ is a link, even if the $\eta_i$ are null-homotopic, in general some condition on $J$ is necessary. Examples are given in Section \ref{section:examples}.

Theorem~\ref{mainthm} gives a very general method to prove that links are slice links. Yet the theorem applies only to links obtained by multi-infection starting from a known slice link, which a priori seems like a very special class. In fact, up to smooth concordance, it is not a restrictive class. The following observation, proven in Section~\ref{section:generality}, shows that, up to smooth concordance, \emph{every} algebraically slice knot can be obtained from a ribbon knot $L$ by multi-infection on a set of curves $\{\eta_i\}$ that lie in the commutator subgroup of $\pi_1(S^3\sm L)$ (such $\eta_i$ are at least candidates to be null-homotopic in the exterior of some set of slice disks for $L$). 

\begin{proposition}\label{prop:generality} Suppose $\mathcal{L}$ is any algebraically slice boundary link (for example any algebraically slice knot). Then $\mathcal{L}$ is smoothly concordant to a link $I$ which is of the form $I(L,J,\E_\varphi)$ where $L$ is a ribbon link, $J$ is a string link with linking numbers zero and the $\eta_i$ lie in the intersection of the terms of the lower central series of $\pi_1(S^3\sm L)$.
\end{proposition}

Note that, by Stallings' theorem, a curve $\eta_i$ that is null-homotopic in the exterior of some set of slice disks must lie in the intersection of the lower central series of the link group.

\section{Proof of Theorem \ref{mainthm}}


\subsection{A sliceness criterion}

We start the proof of Theorem \ref{mainthm} by recalling the
following well--known criterion for links  that asserts that a link $L$ is slice if and only if $M_L$, the $3$-manifold obtained from $L$ by zero-framed Dehn surgery,
is the boundary of a $4$-manifold meeting certain homological criteria. The strategy of our proof will be to construct such a $4$-manifold for the zero surgery on the link $I(L,J,\E_\varphi)$ obtained by infection as in Theorem \ref{mainthm}.

 \begin{proposition} \label{prop:topslice}
A link $L=L_1\amalg\dots \amalg L_m$ is slice if and only if there exists a 4--manifold $W$ such that \bn
\item $\partial W=M_L$,
\item $\pi_1(W)$ is normally generated by images of meridians of $L$,
\item $H_1(W)\cong \Z^m$,
\item $H_2(W)=0$.
\en
\end{proposition}

\begin{proof}
Let $L=L_1\amalg\dots \amalg L_m$ be a link in $S^{3}$ and let
$D=D_1\amalg\dots\amalg D_m$ be a union of slice disks for $L$. Let $W:=D^4\sm
\nu D$ where $\nu D$ is a tubular neighborhood of $D$, which exists
because $D$ is assumed to be locally flat. It is easy to see that
$W$ satisfies the required properties.

Conversely, given such $W$  we add a 2--handle  to $W$ along a meridian of each component of $L$
and call the resulting manifold $W'$. Using the properties (1) to (4) we can easily see that
$\partial W'=S^3$, $\pi_1(W')=0$, and $H_2(W')=0$. Hence $W'$ is homeomorphic
to $D^4$ by Freedman's solution of the topological Poincar\'{e} conjecture in dimension~4. Moreover the cocores of the 2--handles in $W'$ form a disjoint union of  slice disks for the
components of $L$.
\end{proof}

\subsection{Multi--infections and bounding 4--manifolds}\label{section:bound}

The starting point of Theorem~\ref{mainthm} is a slice link $L$. Thus $M_L$ is the boundary of a $4$-manifold, $W_L\equiv D^4\sm \nu D_1\amalg \dots\amalg \nu D_r$, that satisfies the properties of that Proposition~\ref{prop:topslice}. Our goal is to produce a $4$ manifold whose boundary is $M_{I(L,J,\E_\varphi)}$ that satisfies these properties. This will establish that
$I(L,J,\E_\varphi)$ is slice. In this subsection, as a preliminary step we will exhibit a canonical cobordism between $M_L$, $M_{I(L,J,\E_\varphi)}$ and a third manifold $M_{\hat{J}}$, the zero surgery on the link obtained by closing up the string link $J$, as shown in Figure~\ref{stringlink}.

\begin{figure}[h] \centering
\begin{tabular}{cc}
\includegraphics[scale=0.3]{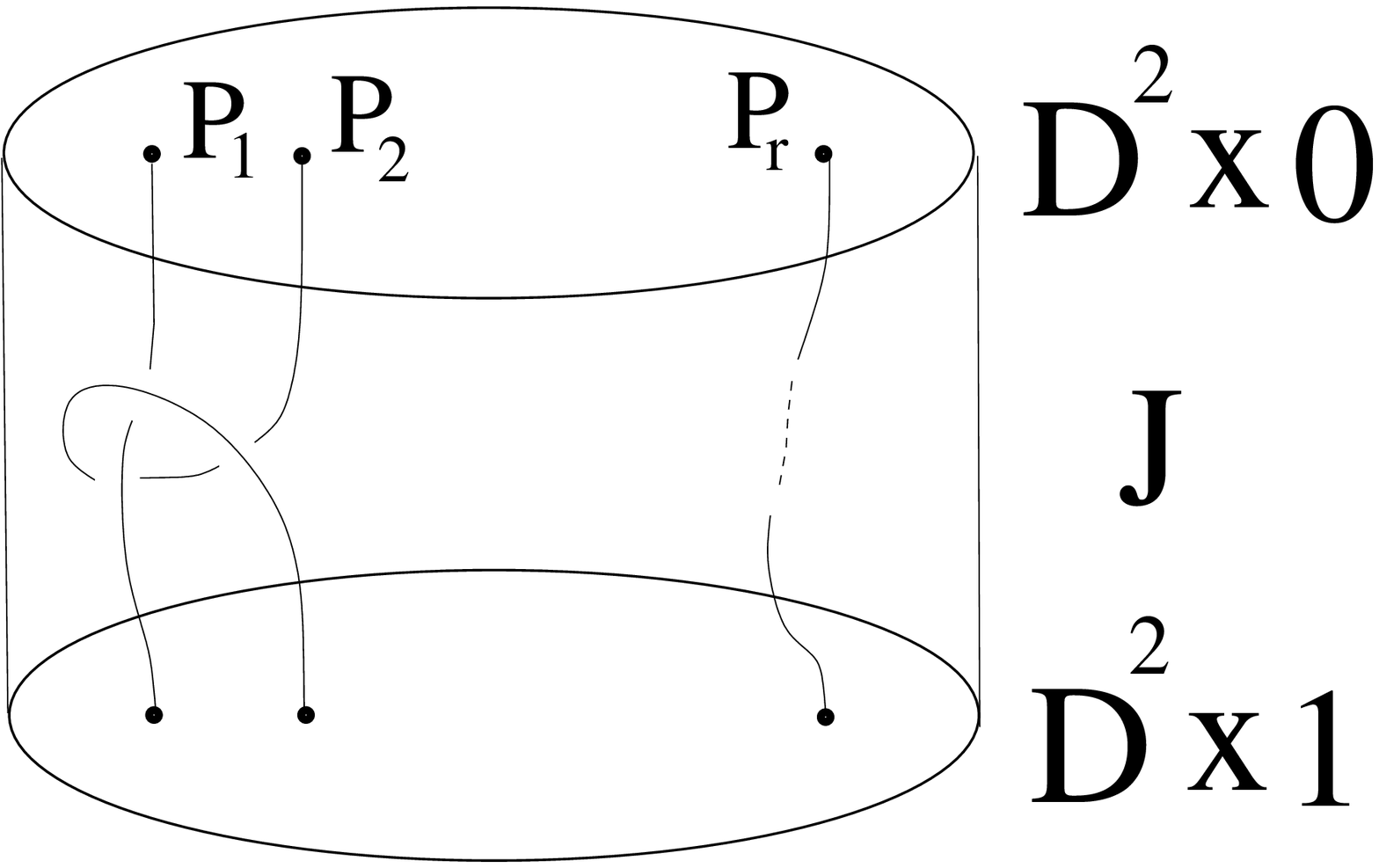}&\hspace{2cm}\includegraphics[scale=0.3]{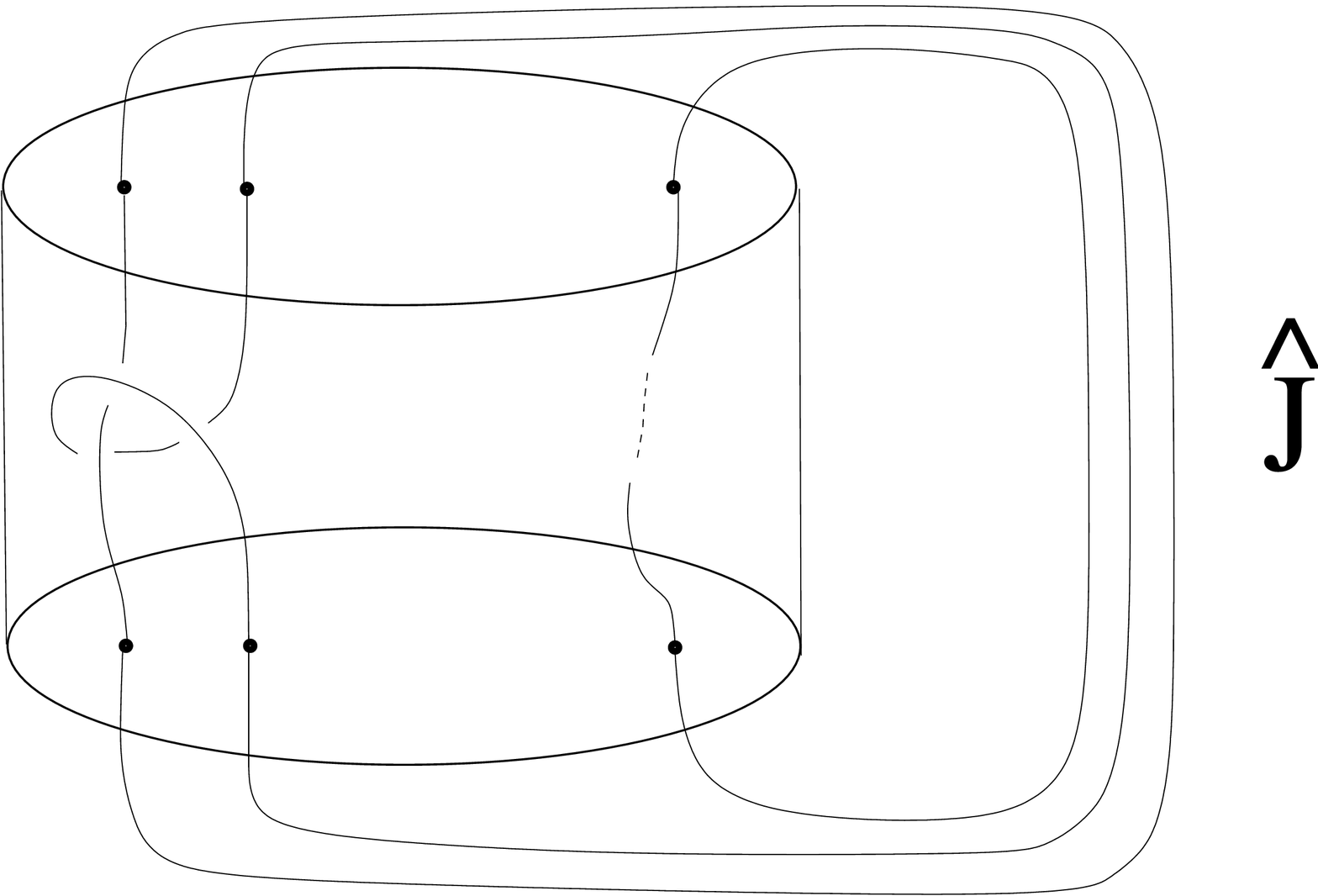}
\end{tabular}
\caption{String link and its closure.} \label{stringlink}
\end{figure}

First we give a more formal definition of the
multi--infection of a link. Let  $L\subset S^3$ be an arbitrary link and $\varphi:\Bbb{E}\to S^3$ be a proper map of an $r$--multi--disk. Recall that $\E_\varphi$ is the image of that disk and we denote by $E_\varphi$ the complement of the $r$ subdisks in $\E_\varphi$.
Let $J$ be an arbitrary $r$-component  string link as in Figure~\ref{stringlink}. Note that  $(\E_\varphi \sm E_\varphi)$ (as shown in Figure~\ref{multi_infection}) is a $2$-disk with $r$ subdisks deleted and so $(\E_\varphi \sm E_\varphi)\times [0,1]$ (as shown in the center of Figure~\ref{fig:infection})
may be viewed as a copy of the exterior of the trivial $r$-component string link. This  manifold has the same boundary as the exterior of the $r$-component string link $J$, denoted $\big (\dtimes\sm \nu J\big)$. Thus we can alter $S^3$ (in the complement of $L$) by deleting the exterior of this trivial $r$-string link and inserting the exterior of the (nontrivial) string link $J$. This should be done in such a way as to equate the meridians and longitudes of these two string links. Recall that the meridians of the trivial string link are the boundary curves of $\varphi(E_1),\dots,\varphi(E_r)$ that we denote by $\eta_1,\dots,\eta_r$. We claim that the resulting manifold is homeomorphic to $S^3$ since
\[ \ba{cl} &\big(S^3\sm \int ((\E_\varphi \sm  E_\varphi)\times [0,1])\big)\dcup \big(\dtimes\sm \nu J\big)\\
=&\big(S^3\sm \E_\varphi \times [0,1]\big)\dcup \big((\dtimes\sm \nu J)\cup
(\overline{E_\varphi}\times
[0,1]) \big)\\
\cong &S^3.\ea
\]
The last homeomorphism follows from the observation that the previous space is the union of two
3--balls. Finally we define the link $I(L,J,\E_\varphi)$ to be the image of the link $L$ under this homeomorphism. It is easy to see that this formal definition agrees with the  more intuitive definition
in the introduction. In the sequel, we often abbreviate $I(L,J,\E_\varphi)$ by $I$.

This definition yields a description of the multi--infection as: deleting the exterior of a trivial string link and inserting the exterior of a non-trivial string link. Since this deletion/insertion occurs in the complement of $L$, it applies equally to the zero-framed surgery manifolds $M_L$ and $M_I$. That is
$$
M_{I} = \big(M_L \sm \{\text{exterior of trivial string link}\}\big)\dcup \big(\text{exterior of} ~J\big).
$$


From now on assume that we are in the situation of Theorem \ref{mainthm} where $L$ is a slice link and let $W_L=D^4\sm \nu D_1\amalg \dots\amalg \nu D_r$. Recall that $\partial W_L=M_L$ contains a copy of the exterior of the trivial $r$-string link, the handlebody
$\mathbb{H}\equiv (\E_\varphi \sm  E_\varphi)\times [0,1]$ as in the center of Figure~\ref{fig:infection}. Furthermore we claim that $M_{\hat{J}}$ also contains a canonical copy of $\mathbb{H}$. In fact $M_{\hat{J}}$ decomposes as the union of the exterior of the string link $J$ and the exterior of a trivial $r$-component string link.  To see this, view $D^2\times
[0,1]$ as a submanifold of $S^3$ via the standard embedding and  let $B^2\times
[0,1]$ denote the complementary $3$-ball. View the string link $J$ as contained in $D^2\times
[0,1]$ and regard the remainder of $\hat{J}$ as a trivial string link, $T$, contained in $B^2\times
[0,1]$. Clearly then
$$
(S^3\sm\nu\hat{J})=~(\dtimes \sm \nu J)\cup (B^2\times
[0,1] \sm \nu T).
$$
Since
$$
M_{\hat{J}}=\big(S^3\sm \nu \hat{J}\big)\dcup \big(\cup_{i=1}^r \mu_{\hat{J}_i}\times D^2\big),
$$
$M_{\hat{J}}$ decomposes into $(\dtimes \sm \nu J)$, the exterior of the string link $J$, and the handlebody
$$
(B^2\times [0,1] \sm \nu T)\dcup \big(\cup_{i=1}^r \mu_{\hat{J}_i}\times D^2 \big) \cong(B^2\times [0,1] \sm \nu T)
$$
The last homeomorphism follows from the fact that each $\mu_{\hat{J}_i}\times D^2$ is attached only along $\mu_{\hat{J}_i}\times A$ where $A$ is an arc in $\partial D^2$. Namely, it is the arc running along $T$, rather then $J$.  It follows that the fundamental group of this handlebody
is the free group on $\mu_1,\dots,\mu_r$, the meridians of $T$ and $\hat{J}$. We now form a $4$-manifold
\[ N=W_L\cup (M_{\hat{J}}\times [0,1])\]
as shown schematically in Figure~\ref{fig:mickey},
\begin{figure}[htbp]
\setlength{\unitlength}{1pt}
\begin{picture}(200,160)
\put(-35,15){\includegraphics[height=150pt]{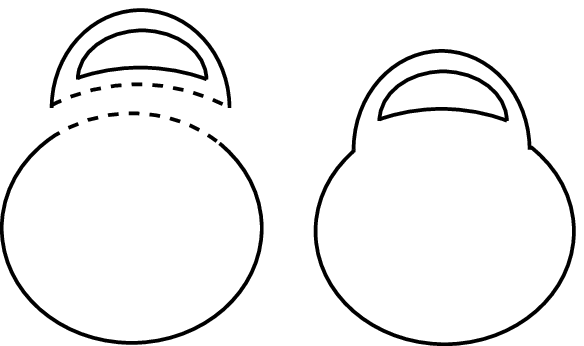}}
\put(18,30){$W_L$}
\put(-68,135){$M_{\hat{J}}\times [0,1]$}
\put(19,98){$\mathbb{H}$}
\put(153,106){$M_{\hat{J}}$}
\put(210,98){$M_I$}
\put(153,30){$N$}
\end{picture}
\caption{}\label{fig:mickey}
\end{figure}
\noindent by identifying $\mathbb{H}$, the copy of the trivial string link exterior in $\partial W_L$ with the copy in $M_{\hat{J}}\times \{0\}$ (shown dashed in Figure~\ref{fig:mickey}) in a way such that the curves $\eta_i$ on the former get identified to the meridians $\mu_{\hat{J}_i}$ of the latter. This is done in such a way that the ``new'' boundary component created is precisely $M_I$ since it is obtained from $M_L$ be deleting the trivial string link exterior and inserting the exterior of $J$.

A key observation is that the curves $\eta_i$ which are equated to the meridians $\mu_i$ of $J$ live in $\mathbb{H}\subset \partial W_L$ and \emph{are null-homotopic in $W_L$} by hypothesis.

If $\hat{J}$ were itself a slice link then we would know that $M_{\hat{J}}$ were the boundary of some $4$-manifold $W$ that satisfies the conditions of Proposition
\ref{prop:topslice}. We could then use this $W$ to cap off $M_{\hat{J}}\subset \partial N$, resulting in $4$-manifold $N'$ whose boundary is $M_I$ and which satisfies the conditions of Proposition~\ref{prop:topslice}, proving that the infected link $I$ were slice. This establishes the following (previously known) very special case of Theorem~\ref{mainthm} which holds without any hypotheses on the curves $\eta_i$.

\begin{corollary} The link obtained by a multi-infection of a slice link $L$ using a string link $J$ whose closure is a slice link, is again a slice link.
\end{corollary}

However, in general $M_{\hat{J}}$ will not bound a $4$-manifold that satisfies Proposition
\ref{prop:topslice}. In this case we must be more clever and make use of our hypotheses on the $\eta_i$ curves.

\begin{lemma} \label{lem:propN} $N$ satisfies the following conditions:
\begin{itemize}
\item [1.] $\partial(N)=M_I\amalg -M_{\hat{J}}$,
\item [2.] $\pi_1(N)$ is normally generated by the meridians of $I$,
\item [3.] $H_1(M_I)\to H_1(N)$ is an isomorphism,
\item [4.] $H_2(M_{\hat{J}})\to H_2(N)$ is an isomorphism,
\item [5.] $\pi_1(M_{\hat{J}})\to \pi_1(N)$ is the zero map.
\end{itemize}
\end{lemma}

\begin{proof}
$N$ is the union of $W_L$ and $M_{\hat{J}}\times [0,1]$ glued along $\mathbb{H}$. Note also that $\{\mu_i\}$ is a basis for the first homology of $\mathbb{H}$. Therefore the Mayer--Vietoris sequence
becomes
\[ \ba{cccccccccccccccccccccc}
&0&\to&H_2(W_L)\oplus H_2(M_{\hat{J}}\times [0,1])&\to& H_2(N)\\
\to& \oplus_{i=1}^r \Z\mu_{i}&\to&H_1(W_L)\oplus
H_1(M_{\hat{J}}\times [0,1])&\to& H_1(N)&\to& 0.\ea \]
Since $
 \oplus_{i=1}^r \Z\mu_{\hat{J}_i}\longisoarrow H_1(M_{\hat{J}})$ and since $H_2(W_L)=0$,
 it follows that $H_2(M_{\hat{J}})\to H_2(N)$ is an isomorphism, establishing $(4)$. Since $\mu_i=\eta_i$ dies in $H_1(W_L)$ it also follows
  that $H_1(W_{L})\to H_1(N)$ is an isomorphism. But $H_1(W_{L})\cong H_1(M_L)\cong \mathbb{Z}^m$ generated by the meridians of
  $L$. Clearly these same meridians are a basis for $H_1$ of the infected link exterior and thus for $H_1(M_I)$. This establishes $(3)$.

In order to prove $(5)$ note that the map
$$
\langle\mu_i\rangle \cong \pi_1(\mathbb{H})\to \pi_1(M_{\hat{J}}\times [0,1])
$$
is surjective and the map
$$
\langle \eta_i\rangle \cong \pi_1(\mathbb{H})\to \pi_1(W_L)
$$
is the zero map. When gluing the two copies of $\mathbb{H}$, the
meridians $\mu_i$ are identified with the $\eta_i$, establishing $(5)$. By the Seifert-van Kampen theorem we have
\[ \pi_1(N)=\pi_1(W_L)*_{\{\eta_i=\mu_i\}}\pi_1(M_{\hat{J}}).\]
Moreover $\pi_1(W_I)$ is normally generated by the meridians of the link $I$, and $\pi_1(M_{\hat{J}})$ is normally generated by $\{\eta_i\}$ which are trivial in $\pi_1(W_L)$. Thus
$\pi_1(N)$ is normally generated by
the meridians of $I$ establishing $(2)$.
\end{proof}

\subsection{Conclusion of the proof}
We show how  the proof of a theorem of Freedman and Teichner can be used to alter $N$ to a $4$-manifold, $N'$, whose boundary is $M_I$ and which satisfies Proposition~\ref{prop:topslice}. We strongly encourage the reader to have pages 547--549 of \cite{FnT95} available.

Recall the situation shown on the right-hand side of Figure~\ref{fig:mickey}. Let $M$ denote a collar on the $M_{\hat{J}}$ boundary component as indicated by the shaded portion on the left-hand side of Figure~\ref{fig:M1}.
\begin{figure}[htbp]
\setlength{\unitlength}{1pt}
\begin{picture}(200,160)

\put(23,30){$N$}
\put(-32,135){$M_I$}
\put(9,98){$\delta_i$}

\put(170,27){$N\sm M_1$}
\put(-35,15){\includegraphics[height=150pt]{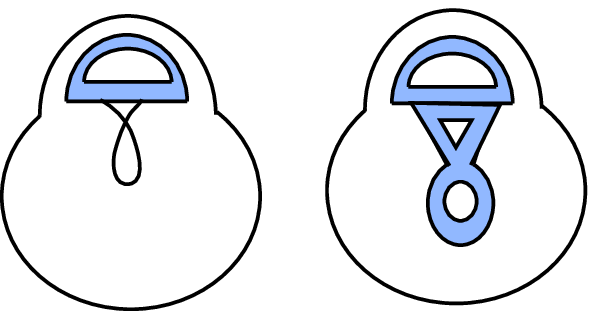}}
\put(20,133){$M_{\hat{J}}$}
\put(210,98){$\partial^+ M_1$}
\end{picture}
\caption{}\label{fig:M1}
\end{figure}

\noindent A note of caution is in order. We shall shortly appeal to the details of a proof in \cite{FnT95}. In that proof, the $M_{\hat{J}}$ -boundary component of $N$ is capped off by a $4$-manifold that is called $M$. But in fact this ``cap'' is not important to the proof (since the strategy is to replace it anyway). Therefore we omit the cap. Our collar $M$ will play the role of $M$ and our $N$ will play the role of $N$ in \cite{FnT95}.

Let $\partial ^+ M$ denote the ``outer'' boundary component of the collar $M$. Recall that $\pi_1(\partial ^+ M)=\pi_1(M_{\hat{J}})$  is normally generated by its meridians $\mu_i=\eta_1,\dots,\mu_r=\eta_r$ and by assumption these curves bound immersed disks $\delta_i$ in $W_L$ where $c$ is the total number of intersections and self-intersections. One such disk is shown schematically on the left side of Figure~\ref{fig:M1}. We now closely follow the proof of ~\cite[p.~547]{FnT95} using the same notation. In accordance with that notation, set $\gamma_i=\eta_i=\mu_i$. Let $M_1$ be a regular neighborhood  of $M\cup \{\delta_i\}\hookrightarrow N$ as shown schematically on the right-hand side of Figure~\ref{fig:M1} by the shaded portion of $N$. Now discard $M_1$ and consider $N\sm M_1$, the unshaded part of the figure. The latter has a new boundary component, $\partial ^+ M_1$. The strategy is to produce, using the proof in \cite{FnT95}, another $4$--manifold $M_3$ with $\partial M_3=\partial^+M_1$ and use it to plug up this hole in $N\sm M_1$. Then, letting
\[
N'=(N\sm M_1)\cup_{\partial M_3} M_3,
\]
we see that $\partial N'=M_I$ and we will verify that $N'$ also satisfies the other conditions of Proposition~\ref{prop:topslice}, establishing that $I$ is a slice link.

\begin{lemma}\label{lem:M_3} There exists a $4$-manifold $M_3$ with $\partial M_3=\partial^+M_1$ such that
\begin{itemize}
\item [1.] The inclusion of the boundary induces an isomorphism $H_1(\partial M_3)\cong H_1(M_3)$.
\item [2.] $M_3$ is homotopy equivalent to a wedge of $c$ circles where these circles correspond precisely to the double point loops among the $\delta_i$.
\end{itemize}
\end{lemma}

Before constructing $M_3$, we prove that its existence will enable us to finish the proof of Theorem~\ref{mainthm}.

\begin{lemma}\label{lem:nminusm} Using the inclusion induced maps, the following statements hold:
\begin{itemize}
\item [1.] $H_1(N\sm M_1)\to H_1(N)$ is an isomorphism,
\item [2.] $H_2(\partial^+M_1)\to H_2(N\sm M_1)$ is surjective,
\item [3.] $\pi_1(N\sm M_1)$ is normally generated by the meridians of $I$ and the meridians of the disks $\delta_i$.
\end{itemize}
\end{lemma}
\begin{proof} First note that, since $M$ is a collar, $N\sm M_1 \cong N\sm \cup_i \nu\delta_i$, where $\nu\delta_i$ is a (closed) regular neighborhood of $\delta_i$. Then excision and Poincar\'e duality give  isomorphisms
$$
H_p(N,N\sm M_1) \cong H_p(N,N\sm \cup_i \nu\delta_i)\cong H_p(\cup_i\nu\delta_i,\partial'(\cup_i\nu\delta_i))\cong H^{4-p}(\cup_i\nu\delta_i,\cup_i\nu(\partial\delta_i))
$$
where we have decomposed the boundary $\partial(\cup_i\nu\delta_i)$ of the regular neighborhood into the two relevant parts. The latter groups are given by
\[
H_p(N,N\sm M_1) \cong H^{4-p}(\cup_i\nu\delta_i,\cup_i\nu(\partial\delta_i)) \cong
\begin{cases}
\Z^r &  \quad \text{ if } \quad p=2, \\
\Z^{c} &  \quad \text{ if } \quad p=3, \\
0 &  \quad \text{ else. } \quad
\end{cases}
\]
For $p=2$, generators are given by transverse disks to the $\delta_i$ and for $p=3$, each intersection point $P$ contributes a generator via a solid torus $T_P$ in a small neighborhood of $P$ (whose boundary is the well known Clifford torus and which intersects the double point loop exactly once). Thus the long exact sequence of the pair $(N,N\sm M_1)$ becomes
$$
\Z^c \to H_2(N\sm M_1)\to H_2(N)\overset{\pi}\to \mathbb{Z}^r\to H_1(N\sm M_1)\to H_1(N)\to 0,
$$
where $\pi$ is given by the algebraic intersection numbers with the various $\delta_i$. Thus the composition of $H_2(M_{\hat{J}})\cong H_2(N)$ (see ($4$) of Lemma~\ref{lem:propN}) with $\pi$ is given by the matrix  of intersection numbers of capped-off Seifert surfaces for $\hat{J}_i$ with the $\partial\delta_j=\gamma_j$. Since $\gamma_j$ is a meridian of $\hat{J}_j$, this matrix is the identity with respect to these bases (we have used that the linking numbers of $\hat{J}$ are zero). Thus $\pi$ is an isomorphism and ($1$) above follows. It also follows that $H_2(N\sm M_1)$ is generated by the Clifford tori $\partial T_P$ and since these clearly lie in $\partial^+M_1$, statement ($2$) also follows.

For ($3$), recall from property ($2$) of Lemma~\ref{lem:propN} that $\pi_1(N)$ is normally generated by the meridians of $I$. Any homotopies in $N$ may be assumed to hit $\delta_i$ transversely, so ($3$) follows immediately.

\end{proof}

Now, assuming we have constructed $M_3$ as in Lemma~\ref{lem:M_3}, we claim:

\begin{lemma}\label{lem:N'} $N'=(N\sm M_1)\cup_{\partial M_3} M_3$ satisfies the conditions of Proposition~\ref{prop:topslice}.
\end{lemma}
\begin{proof} Consider the Mayer-Vietoris sequence for $N'=(N\sm M_1)\cup M_3$:
$$
H_1(\partial ^+M_1)\overset{\psi}\to H_1(N\sm M_1)\oplus H_1(M_3)\to H_1(N')\to 0.
$$
By property ($1$) of Lemma~\ref{lem:M_3}, $H_1(\partial ^+M_1)\cong H_1(M_3)$. It follows that $\psi$ is injective and that
$$
H_1(N\sm M_1)\cong H_1(N'),
$$
and so by properties ($3$) of Lemma~\ref{lem:propN} and ($1$) of Lemma~\ref{lem:nminusm}
$$
H_1(M_I)\cong H_1(N)\cong H_1(N\sm M_1)\cong H_1(N').
$$
This establishes condition ($3$) of Proposition~\ref{prop:topslice}. Moreover examining Mayer-Vietoris again:
$$
H_2(\partial ^+M_1)\to H_2(N\sm M_1)\oplus H_2(M_3)\to H_2(N')\to H_1(\partial ^+M_1)\overset{\psi}\to
$$
where $\psi$ is injective and $H_2(M_3)=0$ by property ($2$) of Lemma~\ref{lem:M_3}. Thus
$H_2(N\sm M_1)\to H_2(N')$
is surjective. Thus by property ($2$) of Lemma~\ref{lem:nminusm},
$$
H_2(\partial^+M_1)\to H_2(N\sm M_1)\to H_2(N')
$$
is surjective. Since any class in $H_2(N')$ is carried by $\partial^+M_1=\partial M_3$ and $H_2(M_3)=0$, it follows that  $H_2(N')=0$, establishing condition ($4$) of Proposition~\ref{prop:topslice}.

Finally consider $\pi_1(N')$ which, by the Seifert-Van Kampen theorem, equals
$$
\pi_1(N\sm M_1)*_{\pi_1(\partial ^+M_1)}\pi_1(M_3).
$$
The map  $\pi_1(\partial ^+M_1)\to\pi_1(M_3)$
is surjective because the double point loops come from the boundary. Therefore, $\pi_1(N\sm M_1) \to\pi_1(N')$ is also surjective. Property ($3$) of Lemma~\ref{lem:nminusm} implies that $\pi_1(N')$ is normally generated by the meridians of $I$ and the meridians of the disks $\delta_i$. But the meridians of the disks $\delta_i$ live on the Clifford tori and hence intersect trivially with the solid tori $T_P$ from Lemma~\ref{lem:nminusm}. In the construction of $M_3$ it will become clear that intersections with $T_P$ give the isomorphism of $\pi_1M_3$ with the free group on $c$ generators. Therefore, the meridians to $\delta_i$ map trivially to $\pi_1M_3$
and thus $\pi_1(N')$ is normally generated by the meridians of $I$ alone.
Thus $N'$  satisfies all the conditions of Proposition~\ref{prop:topslice}.
\end{proof}

This concludes the proof that $I$ is slice and hence the proof of Theorem~\ref{mainthm}, modulo the proof of Lemma~\ref{lem:M_3}.

\subsection{Using the proof in \cite{FnT95} to construct $M_3$}

\cite[p.~548]{FnT95} explains how to draw a ``Kirby diagram'' for the $3$-manifold $\partial^+M_1$ as follows. First consider the abstract $4$-manifold obtained from $M$ by adding $r$ $2$--handles along ($\gamma_i,f_i$)~$\subset \partial^ +M=M_{\hat{J}}$ using framings $f_i$ induced from the $\delta_i$. This is \emph{not} embedded in $N$. Let $\Sigma$ denote the resulting homology sphere obtained as the top boundary, i.e. $\Sigma$ is obtained from $M_{\hat{J}}$ by $f_i$--framed surgery on the meridians $\gamma_i$ as shown in part $A$ of Figure~\ref{fig:handleslides} (only one component of $\hat{J}$ is shown). Let $\{m_1,...,m_r\}$ denote a set of meridians for the $\{\gamma_1,...,\gamma_r\}$, also shown in part $A$ of the figure.
\begin{figure}[htbp]
\setlength{\unitlength}{1pt}
\begin{picture}(260,150)
\put(-41,2){$A$}
\put(70,2){$B$}
\put(175,2){$C$}
\put(287,2){$D$}
\put(0,30){$\hat{J}_i$}
\put(315,30){$\hat{J}_i$}
\put(-7,75){$\gamma_i$}
\put(73,83){$0$}
\put(57,75){$f_i$}
\put(0,135){$\mu_i'$}
\put(111,137){$\mu_i'$}
\put(0,105){$m_i$}
\put(-67,85){$f_i$}
\put(-80,48){$0$}
\put(30,44){$0$}
\put(254,45){$0$}
\put(142,45){$0$}
\put(-80,10){\includegraphics[height=140pt]{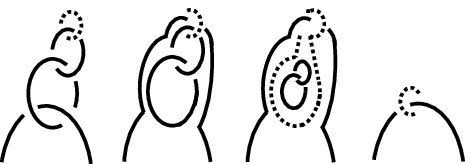}}
\end{picture}
\caption{}\label{fig:handleslides}
\end{figure}

Let $L$ denote the link $\{m_1,...,m_r\}\subset \Sigma$ and $\Sigma_L$ denote the $0$--framed surgery on $L\hookrightarrow \Sigma$. Meridians of this link are called $\mu_i'$ and are shown dashed in part $A$ of the figure. Pictures $A$ through $D$ illustrate a proof of the observation in \cite[p.~547]{FnT95} that $\Sigma_L\cong M_{\hat{J}}=\partial ^+M$ by a map that sends a meridional set $\{\mu_1',...,\mu_r'\}$ to $\{\gamma_1,...,\gamma_r\}$. This observation will be used later. This is seen by first sliding each component $\hat{J}_i$ over the corresponding $0$--framed $m_i$ which results in part $B$ of Figure~\ref{fig:handleslides}. Then the $0$--framed $m_i$ cancels with the $f_i$--framed $\gamma_i$, yielding part $D$ of the figure.

To obtain a description of $\partial ^+ M_1$ from $\Sigma$, one must take into account the self-plumbings of the $\delta_i$.
\begin{lemma}\label{lem:M_2} There exists a $4$-manifold $M_2$ with fundamental group free on $c$ generators and with $\partial M_2=\partial^+M_1$ such that
\begin{itemize}
\item [1.] The inclusion of the boundary induces an isomorphism $H_1(\partial M_2)\cong H_1(M_2)$.
\item [2.] $\pi_2M_2$ is a free $\Z[\pi_1M_2]$-module with hyperbolic intersection form.
\end{itemize}
\end{lemma}
\begin{proof}
In \cite[Figugres~4.1-4.3]{FnT95}, a ``Kirby calculus'' description of $\partial ^+ M_1$ is obtained, viewed as handles attached to the unique contractible $4$-manifold $Z$ with boundary $\Sigma$. Let $M_2$ denote the $4$-manifold given by \cite[Figure~4.3]{FnT95}. It is obtained from $Z$ by attaching $c$ 1-handles and $2c$ 2-handles to $\Sigma=\partial Z$ in a way that  $\partial M_2=\partial M_1$. Moreover, the 2-handles go homotopically trivially over the 1-handles, implying the statement for the fundamental group. Moreover, it follows that $M_2$ is homotopy equivalent to a wedge of $c$ circles and $2c$ 2-spheres, in particular $\pi_2M_2$ is a free $\Z[\pi_1M_2]$-module of rank $2c$. Finally, the figure clearly shows that the 2-handles generate a hyperbolic form on $\pi_2M_2$ which by the homology long exact sequence for the pair $(M_2,\partial M_2)$ implies $(1)$.
\end{proof}

If surgery worked over the free group, we could remove the hyperbolic form on $\pi_2M_2$ to get a manifold $M_3$ with the desired properties of Lemma~\ref{lem:M_3}. We actually just need surgery to work up to s-cobordism (rel. boundary) and this is in fact a theorem in the {\em $\pi_1$-null} case \cite{FQ90}. This condition means that the union of the images of all immersed 2-spheres representing the hyperbolic form maps trivially on $\pi_1$ into the $4$-manifold. In the case of $M_2$, the 2-spheres are made from the cores of the 2-handles, together with null-homotopies of the attaching circles in $Z$. Since $Z$ is simply connected, it suffices to keep those $2c$ null-homotopies {\em disjoint} to make the union of all 2-spheres $\pi_1$-null.

This is where our assumption on the Milnor invariants of $J$ comes in: The above 2-handles are attached to a number of parallel copies of the circles $m_i$ where the total number is precisely $2c$. We now claim that we can replace \cite[Lemma~4.2]{FnT95} by

\begin{lemma}\label{lem:ubar} Any link consisting of $2c$ untwisted parallel copies of the components $m_i$ of $L=\{m_1,...,m_r\}$ in $\Sigma$ bounds  a set of disjointly immersed disks in $Z$. Here we mean that $2c$ is the sum of the number of parallels, $c_i$, of $m_i$.
\end{lemma}

Given this Lemma,  we can eliminate all the second homology of $M_2$ by the $\pi_1$-null surgery theorem up to s-cobordism. This gives the $4$-manifold $M_3$ that is homotopy equivalent to a wedge of circles and satisfies Lemma~\ref{lem:M_3}. The argument works exactly like in the paragraph just below Lemma 4.2 in \cite{FnT95}. This concludes the proof of Lemma~\ref{lem:M_3}, modulo the proof of Proof of Lemma~\ref{lem:ubar}.

\begin{proof}[Proof of Lemma~\ref{lem:ubar}] The Lemma is vacuously true for $c=0$, so assume $c\geq 1$. By ~\cite[Lemma~2.7]{FnT95}, the conclusion of the Lemma is equivalent to the property that
\begin{equation}\label{eq:zero}
\text{All} ~\bar{\mu}-\text{invariants of length less than or equal to} ~2c~\text{vanish for} ~L.
\end{equation}
Milnor's invariants for links in homology $3$-spheres are defined in the exactly same way as for links in $S^3$, see~\cite{FnT95}. Now let $F$ be the free group on $r$ generators and $F\to \pi_1(\Sigma\sm L)$ be the meridional map. The vanishing of the $\bar{\mu}$-invariants of $L$ is equivalent to the following three statements:
\begin{equation}\label{eq:one}
F/F_{2c+1}\cong \pi_1(\Sigma\sm L)/\pi_1(\Sigma\sm L)_{2c+1}
\end{equation}
\begin{equation}\label{eq:two}
F/F_{2c}\cong \pi_1(\Sigma_L)/\pi_1(\Sigma_L)_{2c}
\end{equation}
\begin{equation}\label{eq:2.5}
H_2(\Sigma_L)\to H_2(\pi_1(\Sigma_L)/\pi_1(\Sigma_L)_{2c-1}) ~\text{is the trivial map}.
\end{equation}
The equivalence of ~\ref{eq:zero},~\ref{eq:one} and ~\ref{eq:two} is standard for links in $S^3$ ~\cite{Mi57}. For links in general homology spheres most of this is derived in ~\cite[Section~2]{FnT95}. In particular the equivalence of ~\ref{eq:one} and ~\ref{eq:2.5} is established there using ~\cite[Theorem~1.1]{Dw75}. Now we have reached the key point: the desired property ~\ref{eq:2.5} depends only on the zero surgery $\Sigma_L$, not on $L$ itself (indeed it only depends on $\pi_1(\Sigma_L)$). At this point we only have to recall our previous observation that $\Sigma_L\cong M_{\hat{J}}$. Therefore each of the above conditions is equivalent to the requirement that the $\bar{\mu}$-invariants of length less than or equal to $2c$ vanish for $\hat{J}$. But this was the assumption of our Theorem~\ref{mainthm}.
\end{proof}

\section{Proof of Proposition~\ref{prop:generality}} \label{section:generality}

Suppose $\mathcal{L}$ is an algebraically slice boundary link of $m$ components. We give the proof in the case that $\mathcal{L}$ is a knot. The proof for a boundary link is identical. Since $\mathcal{L}$ is algebraically slice, there is a genus $r$ Seifert surface $\Sigma$ which is in ``disk-band'' form, as suggested by the left-hand side of Figure~\ref{fig:disk-bands}, where the ``$\alpha$-bands'' are untwisted and such that the link $\hat{J}$ formed by the cores of these bands has zero linking numbers. This is possible since we can choose the $\alpha$-bands to generate a metabolizer of the Seifert form.
\begin{figure}[htbp]
\setlength{\unitlength}{1pt}
\begin{picture}(300,80)

\put(-17,48){$\hat{J}_1$}
\put(70,47){$\hat{J}_2$}
\put(277,49){$\hat{J}_2$}
\put(190,48){$\hat{J}_1$}
\put(190,2){$\E$}
\put(-35,10){\includegraphics[height=50pt]{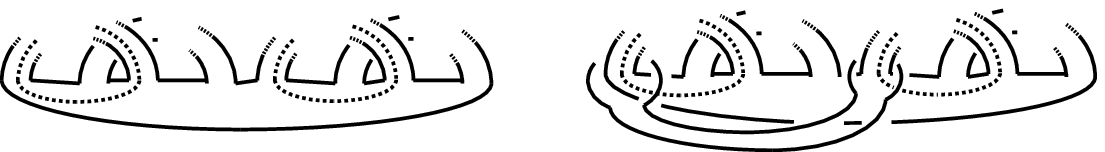}}
\end{picture}
\caption{}\label{fig:disk-bands}
\end{figure}

It is well-known that if $\hat{J}$ is (smoothly) slice then $\mathcal{L}$ is (smoothly) slice (since then the Seifert surface could be ``surgered'' to a disk using the slice disks for $\hat{J}$).
 Let $\varphi:\E\to S^3$ denote an $r$-multi-disk that hits each $\alpha$-band once transversely as suggested on the right-hand side of Figure~\ref{fig:disk-bands}. By thickening up $\E$ we arrive at the local picture shown in the left-most part of Figure~\ref{fig:stringlinks}, of a $2$-cable of the trivial $r$-string link $T$.
\begin{figure}[htbp]
\setlength{\unitlength}{1pt}
\begin{picture}(260,140)
\put(-41,-5){$\mathcal{L}$}
\put(-43,55){$T$}
\put(60,-5){$L$}
\put(165,-2){$I$}
\put(256,-2){$I$}
\put(48,55){$-J$}
\put(145,59){$-J\# J$}
\put(215,65){$\cong$}
\put(251,104){$-J$}
\put(257,47){$J$}
\put(-80,10){\includegraphics[height=140pt]{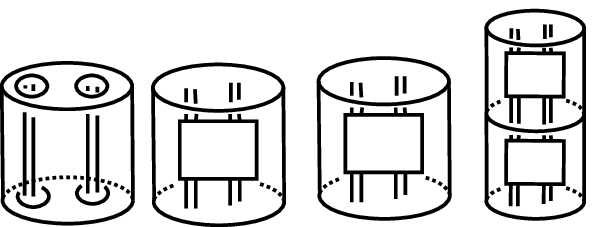}}
\end{picture}
\caption{}\label{fig:stringlinks}
\end{figure}

Let $J$ denote the $r$-string link formed by the cores of the $\alpha$-bands in the \textbf{exterior} of the thickened $\E$. Suppose we replace the $2$-cable of $T$ by the $2$-cable of the string link $-J$ as shown in the second frame of Figure~\ref{fig:stringlinks}, and call the resulting knot $L$. Then $L$ is a knot that admits a disk-band form whose $\alpha$-bands form a link that is the closure of $-J\#J$. Since the latter is a ribbon link, $L$ is a smoothly slice knot (actually a ribbon knot although this takes a little more work to show). On the other hand, suppose we replace $T$ by the $2$-cable of the string link $-J$\#$J$, as shown in the third frame of Figure~\ref{fig:stringlinks}, and denote this knot by $I$. Then $I$ is obtained from the ribbon knot $L$ by a multi-infection using the string link $J$ as indicated by the equivalence of the third and fourth frames of Figure~\ref{fig:stringlinks} (the knot in the fourth frame clearly differs from the knot in the second frame by a tangle insertion-deletion). Moreover since the string links $T$ and $-J$\#$J$ are smoothly concordant, their $2$-cables are also smoothly concordant. It follows that the knot $\mathcal{L}$ is concordant to the knot $I$ (just alter the product concordance from $\mathcal{L}$ to itself by the string link concordance). Thus, the original knot $\mathcal{L}$ is smoothly concordant to the knot $I$ which is obtained from the slice knot $L$ by infection using $J$. The curves $\{\eta_i\}$ are meridians to the $\alpha$-bands and hence lie in the exterior of a system of Seifert surfaces that exhibit $L$ as a boundary link. Then it is well known that they lie in the intersection of the terms of the lower central series of $\pi_1(S^3\sm L)$, since the Seifert surfaces can be used to construct a map to a wedge of circles that sends the $\eta_i$ to the wedge point. This concludes the proof of Proposition~\ref{prop:generality}.

\section{Examples} \label{section:examples}

In this section we give several examples of the applicability of Theorem \ref{mainthm} and Corollary \ref{corthm}. In Section~\ref{sec:intro} we explained how, given any link $\hat{J}$, Theorem \ref{mainthm} could be applied to classes of \emph{links} much more general than Whitehead doubles of $\hat{J}$. In the current section we restrict to the case that $\hat{J}$ is a knot. 

Let $L\subset S^3$ be a link and let $\eta\subset S^3\sm L$ be a closed curve which is the trivial
knot in $S^3$. The curve $\eta$ bounds an embedded disk in $S^3$ which intersects $L$ transversely
and extending this disk so that $\eta$ lies in the interior we get an embedded 1--multi--disk
$\E_\varphi$. Let $\hat{J}$ be a knot and let $J$ be a string knot such that its closure is $\hat{J}$. Recall that, in this case, all of
Milnor's $\bar{\mu}$--invariants of $\hat{J}$ are zero. We can form the infection link $I(L,J,\E_\varphi)$. It is easy to see that this link only depends on $\eta$ and
$\hJ$, and we therefore denote it by $I(L,\hat{J},\eta)$. As we have mentioned, in the literature $I(L,\hJ,\eta)$ is sometimes called the satellite link of $L$ with companion $\hJ$ and axis $\eta$.

\subsection{Infection of ribbon knots by a knot}
In this section we compare Theorem \ref{mainthm} with the two previously known slicing theorems:
\bn
\item If $K$ is a  knot with $\Delta_K(t)=1$, then $K$ is slice (\cite{Fr85}).
\item If $K$ is a knot with $\Delta_K(t)=(2t-1)(t-2)$ and if a certain non--commutative Blanchfield pairing vanishes, then
$K$ is slice (\cite{FlT05}).
\en

We first consider Figure \ref{fig:satknot}. The shaded region in Figure \ref{fig:satknot} (a) is part of a
ribbon disk $D$  for a ribbon knot $K$. Figure \ref{fig:satknot} (b) shows a curve $\eta$ which is
clearly the unknot in $S^3$.
\begin{figure}[h]
\begin{center}
\begin{tabular}{ccccc}
\includegraphics[scale=0.3]{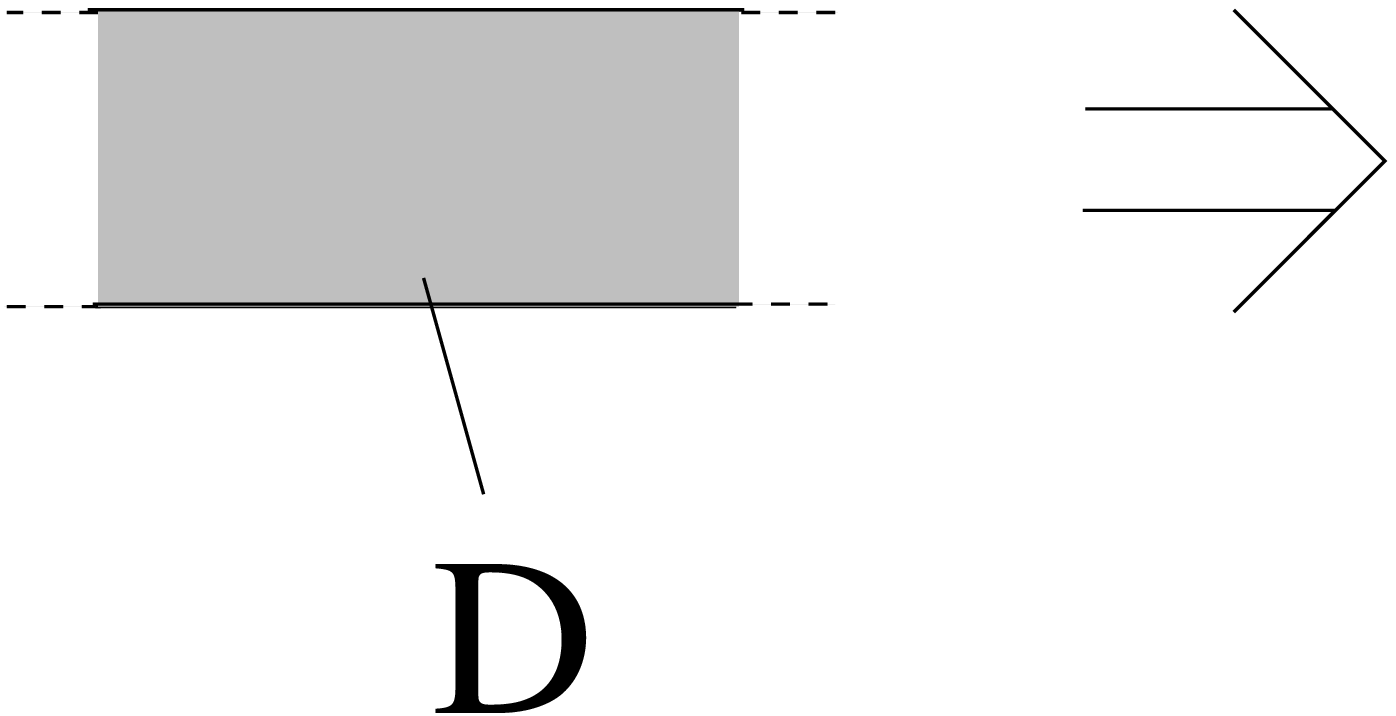}&\hspace{0.2cm}
\includegraphics[scale=0.3]{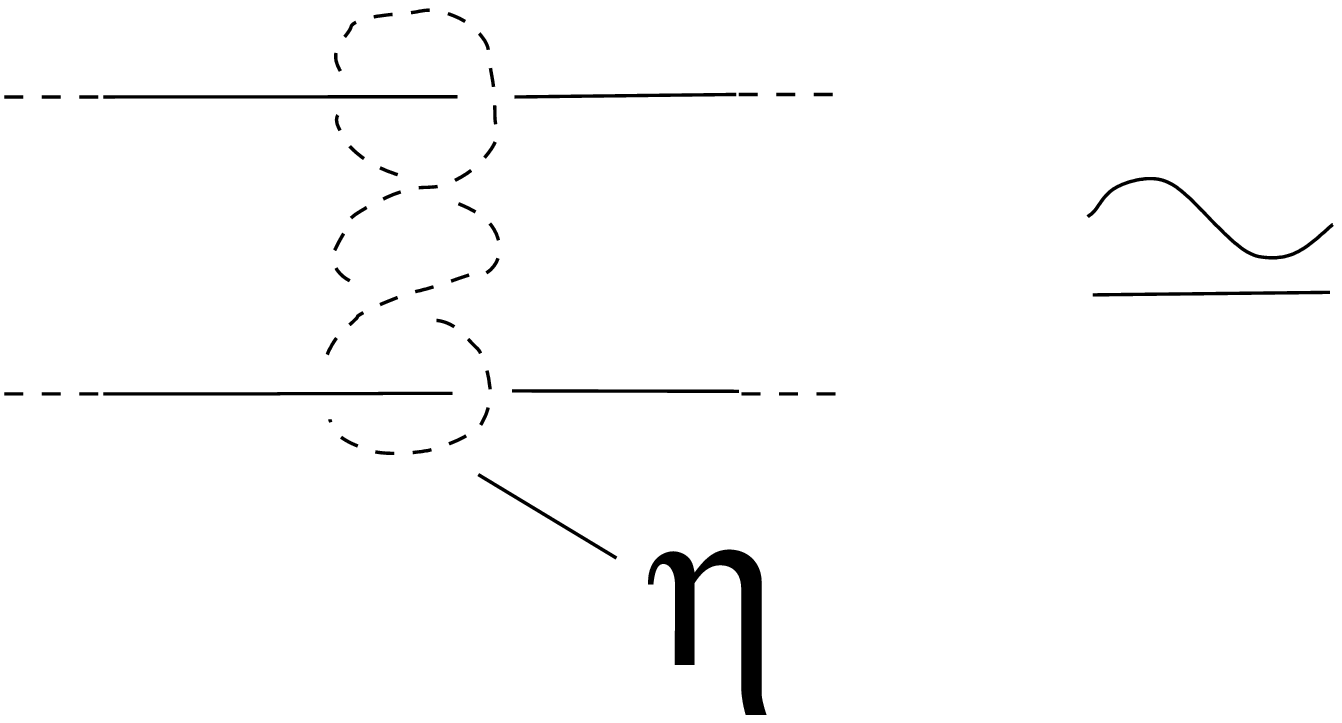}&\hspace{0.2cm}
\includegraphics[scale=0.3]{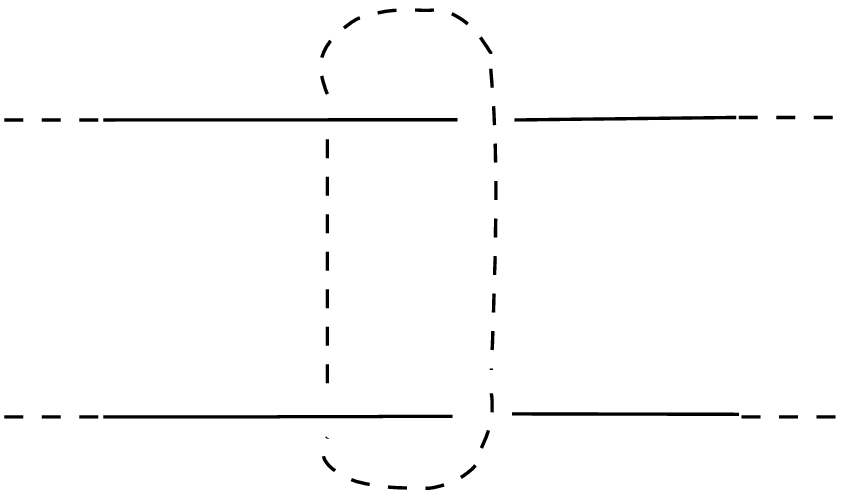}\\
(a)&(b)&(c)
\end{tabular}
\caption{Infection by a knot.} \label{fig:satknot}
\end{center}
\end{figure}
Note that $\eta$ is homotopically equivalent to a curve linking the ribbon disk $D$ once without
intersecting it (cf. Figure \ref{fig:satknot} (c)). It is therefore clear that $\eta$ is homotopically
trivial in $D^4\sm D$. It follows immediately from Corollary \ref{corthm} that $I(K,\hJ,\eta)$ is
slice for any knot $\hJ$.

(1)  Now consider Figure \ref{fig:whiteheaddouble}, it shows two isotopic
pictures for the link $\eta \cup K$.
\begin{figure}[h]
\begin{center}
\begin{tabular}{ccccc}
\includegraphics[scale=0.3]{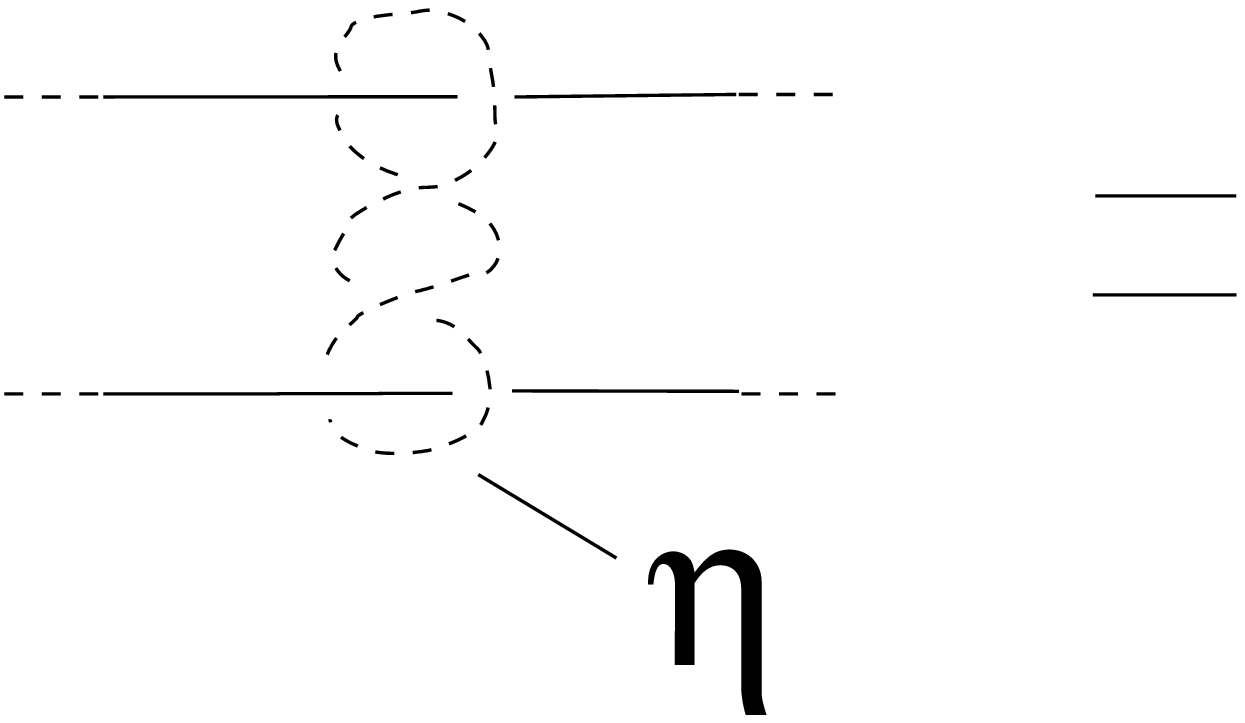}&\hspace{0.2cm}
\includegraphics[scale=0.3]{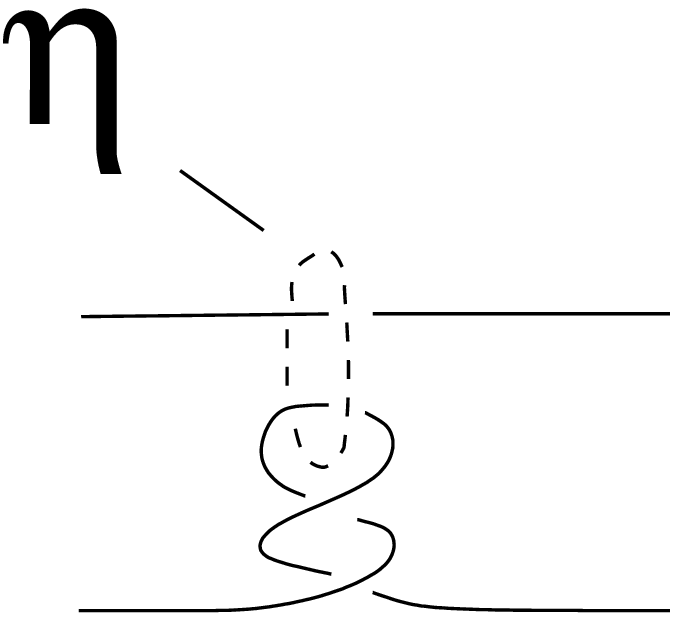}\\
(a)&(b)
\end{tabular}
\caption{Isotopy of $K\cup \eta$.} \label{fig:whiteheaddouble}
\end{center}
\end{figure}
In the special case that $D$ is the trivial ribbon disk for the unknot $K$, it follows immediately from Figure
\ref{fig:whiteheaddouble} (b)  that  $I(K,\hJ,\eta)$ is the Whitehead double of $\hat{J}$. Therefore
Corollary  \ref{corthm} gives another proof that the Whitehead double of any knot $\hat{J}$ is slice.
Note though that there exist many Alexander polynomial knots which are not Whitehead doubles (e.g. the Kinoshita--Terasaka knot), and to which Theorem \ref{mainthm} a priori does not apply.

(2) We consider Figure \ref{example2} (a). The knot $K_1$ is the ribbon knot $6_1$.
\begin{figure}[h]
\begin{center}
\begin{tabular}{ccccc}
\includegraphics[scale=0.3]{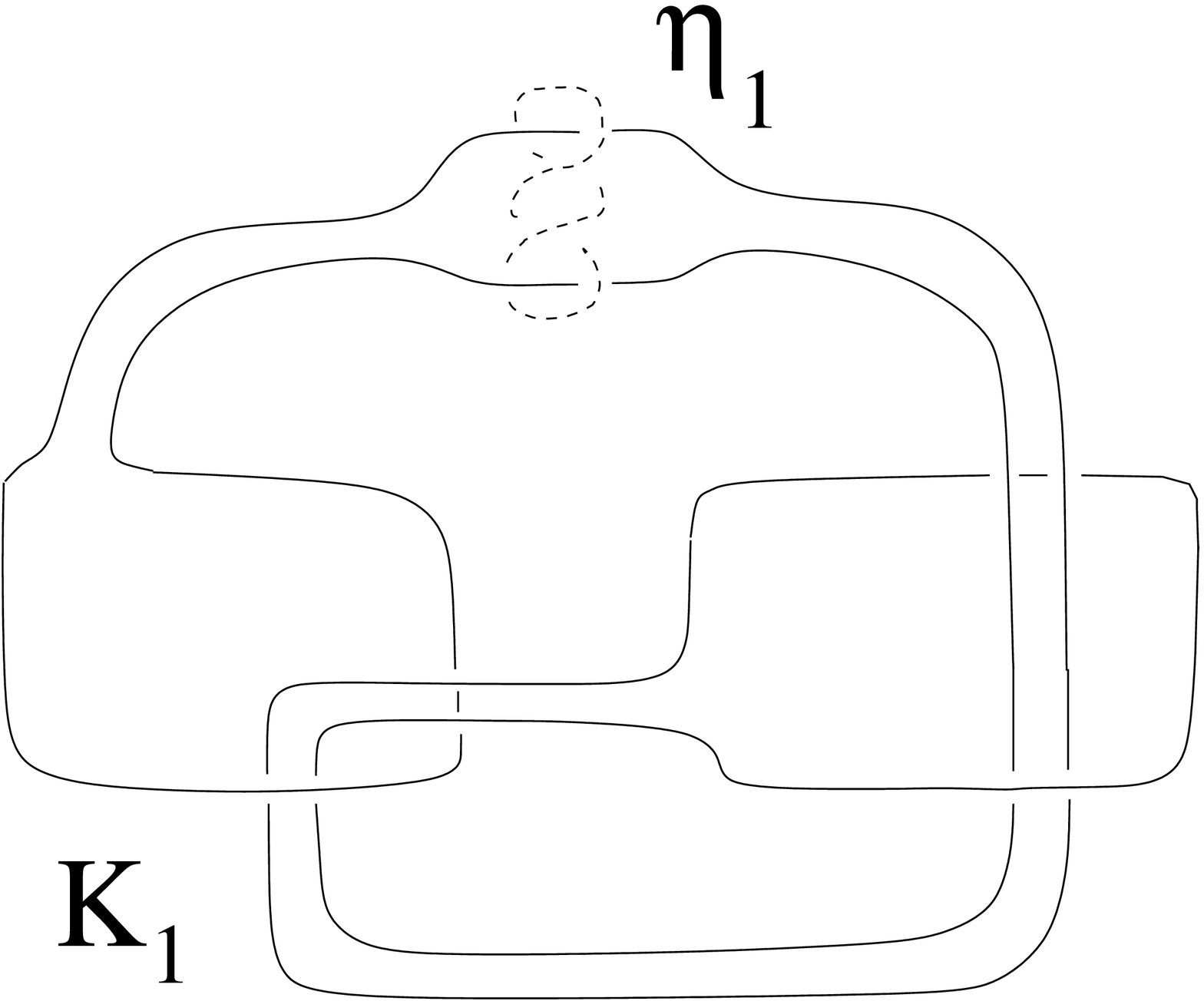}&\hspace{1.5cm}&&&
\includegraphics[scale=0.3]{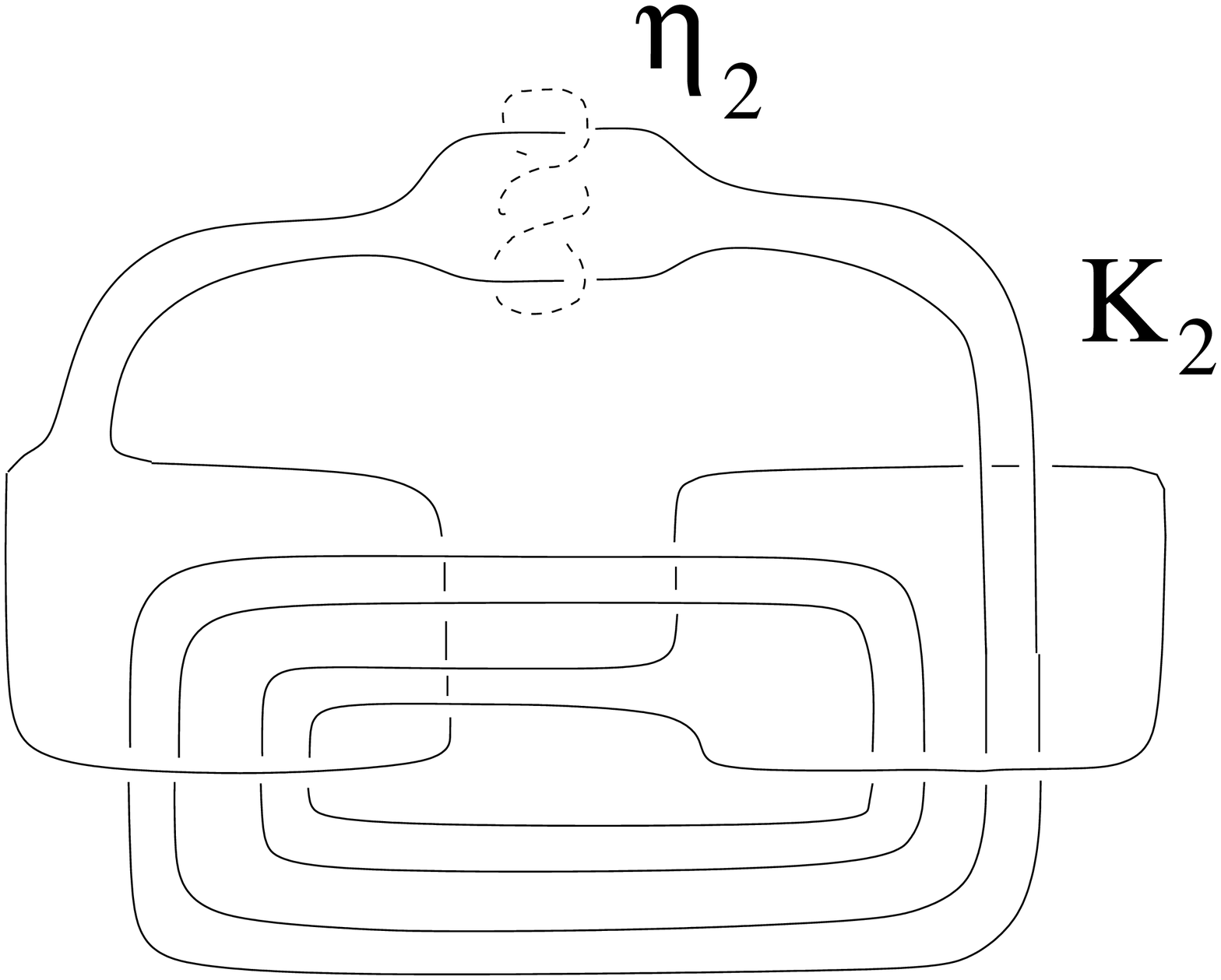}
\end{tabular}
\caption{The knots $K_1$ and $K_2$ with curves $\eta_1$ and $\eta_2$.} \label{example2}
\end{center}
\end{figure}
As we saw above, the knot  $I=I(K_1,\hJ,\eta_1)$ is slice for any knot $\hJ$. Note that $\Delta_I(t)=\Delta_{K_1}(t)=(2t-1)(t-2)$.
In fact \cite[Proposition~7.4]{FlT05} also applies to show that $I$ is slice. We point out that in \cite[Figure~1.5]{FlT05} a different and
\emph{incorrect} curve $\eta$ was chosen (cf. also the correction in \cite{FlT06}).

(3)
We now turn to Figure \ref{example2} (b). By the above discussion  $I=I(K_2,\hJ,\eta_2)$ is slice for any knot $\hJ$.
Note that $\Delta_I(t)=\Delta_{K_2}(t)=(2t-3)(2t^{-1}-3)$. In
particular neither of the two previous slicing theorems can be applied directly.

\begin{remark}
All of the knots in Figure \ref{example2}, $I(K_i,\hJ,\eta_i)$, $i=1,2$, are in fact smoothly concordant to the Whitehead double $Wh(\hat{J})$. Indeed, by ``cutting the ribbon band'' and capping off the trivial component that splits off, one constructs a smooth ribbon concordance from $I(K_i,\hJ,\eta_i)$ to the case in Figure
\ref{fig:whiteheaddouble} (b) where  the knot is the Whitehead double of $\hat{J}$. For many $\hJ$, the knot $Wh(\hat{J})$ is known not to be smoothly slice, in particular for such $\hJ$ the knots $I(K_i,\hJ,\eta_i)$ are slice but \emph{not smoothly slice}.
\end{remark}


\subsection{Satellite links}
We now turn to the study of satellite links. We first point out an array of examples from the literature that illustrate the apparent necessity of the conditions in Theorem~\ref{mainthm}. First we give examples illustrating the necessity of the conditions on the $\eta_i$. If we take $L=L_1\amalg L_2$ to be the trivial link and
$\eta$ as in Figure~\ref{fig:bingdouble}
\begin{figure}[h] \centering
\begin{tabular}{cc}
  \includegraphics[scale=0.3]{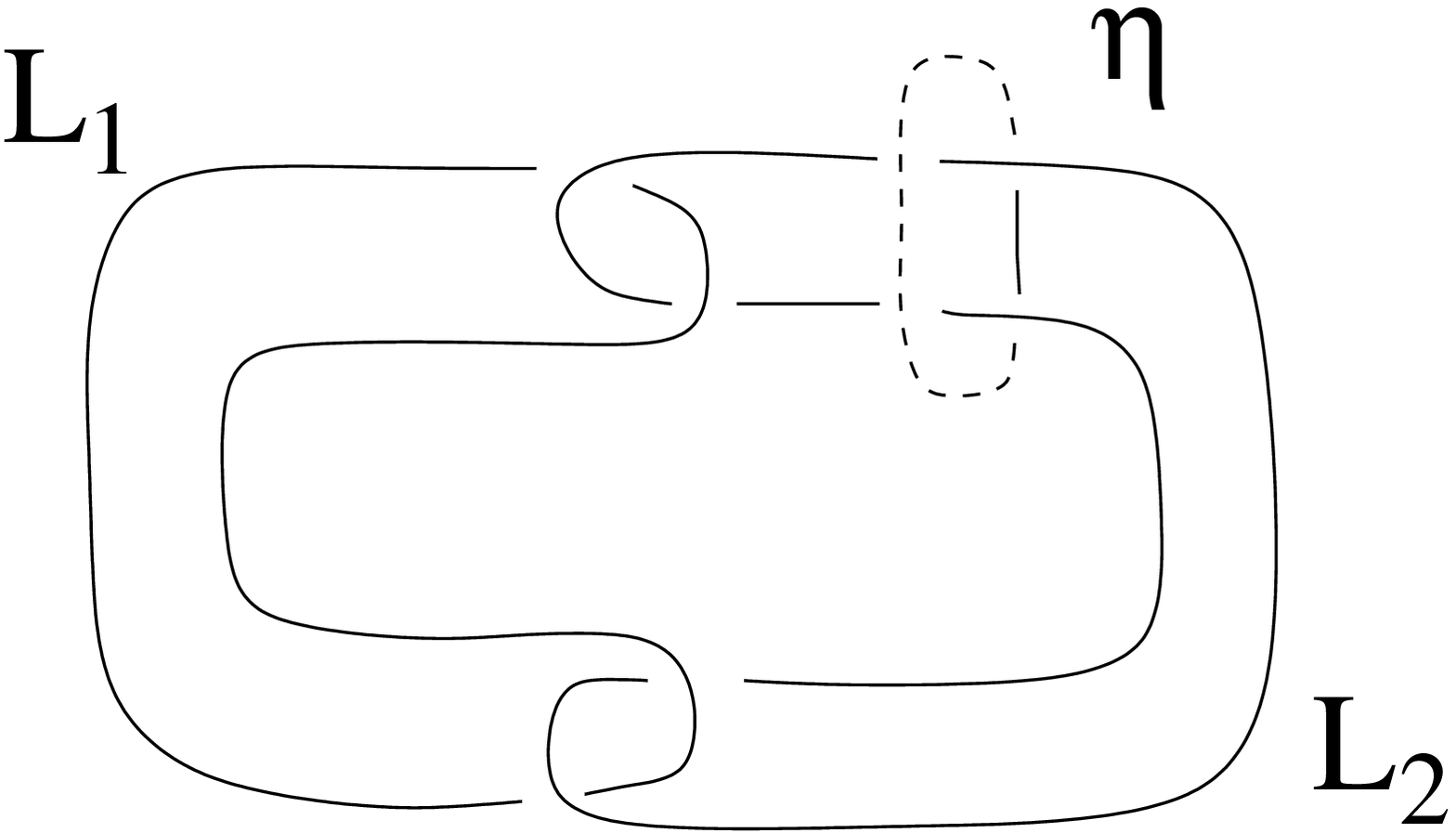}&\hspace{1cm}
\includegraphics[scale=0.3]{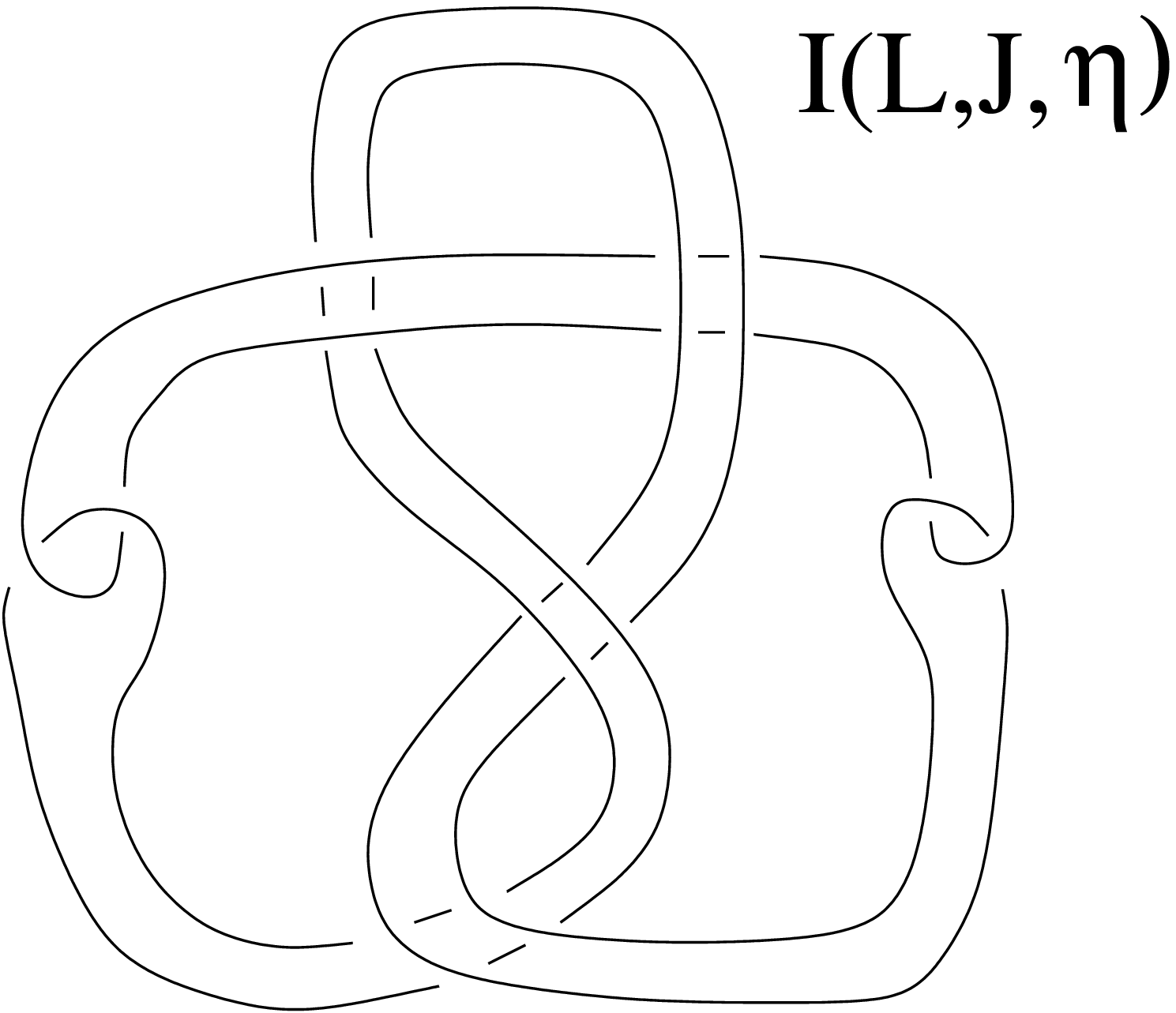}
\end{tabular}
\caption{Satellite construction with $\hJ$ the figure 8 knot, and  $L$ the trivial 2--component
link.} \label{fig:bingdouble}
\end{figure}
then it is easy to see that $\eta$ is a non--trivial element, namely it is the commutator $[x_1,x_2]$ in the free group $\pi_1(S^3\sm L_1\amalg
L_2)=\pi_1(D^4\sm D_1\amalg D_2)=\langle x_1,x_2\rangle$ where $D_1\amalg D_2$ is the obvious slice disk for the trivial
link. Hence Theorem \ref{mainthm} does not apply. Indeed, in many cases it is known that
$I(L,\hJ,\eta)$ (which is just the Bing double $\bing(\hJ)$ of a knot $\hJ$) is not slice. For example
$\bing(\hJ)$ is known to fail to be slice if $\hJ$ is not an algebraically slice knot ~\cite{CLR07}. Even if $\hJ$ is an algebraically slice knot, there are many examples where higher-order signatures obstruct $\bing(\hJ)$ from being a slice link ~\cite{CHL07}. Much more generally, if $\eta$ is taken to be \emph{any} homotopically essential (unknotted) circle in the exterior of the trivial link, then these same invariants obstruct $I(L,\hJ,\eta)$ from being slice ~\cite[Theorem~5.4, Corollary~5.6]{Ha06}~\cite[Theorem~5.8]{CHL07}~\cite{Ch07}. Included among these examples are the so-called iterated Bing doubles of $\hJ$. Therefore for these examples it seems likely that $I(L,\hJ,\eta)$ will be slice if and only if either $\eta$ is null-homotopic or $\hat{J}$ is itself a slice knot.

Even if $\eta$ is null-homotopic, the link $I(L,\hJ,\eta)$ can fail to be slice. If we take $L$  to be the trivial $2$-component link and take $\eta_1,\eta_2$ to be curves as in Figure~\ref{fig:satknot}, then, as previously observed, $I(L,\hJ,\eta)$ is $Wh(\hat{J})$. But the Whitehead double of the Hopf link is known  not to be slice (cf. \cite{Fr88}). Thus some condition on the link $\hJ$ is necessary in general.

On the other hand consider the following very general example.  The shaded region in Figure \ref{fig:satlink}
(a) is part of two ribbon disks $D_1$ and $D_2$ for a link $L=L_1\amalg L_2$. Figure
\ref{fig:satlink} (b) shows a curve $\eta$. If we view $\eta$ as a knot in $S^3$ we see  that the
kinks of $\eta$ cancel with the twists, hence $\eta\subset S^3$ is the unknot in
$S^3$.
\begin{figure}[h]
\begin{center}
\begin{tabular}{ccccc}
\includegraphics[scale=0.3]{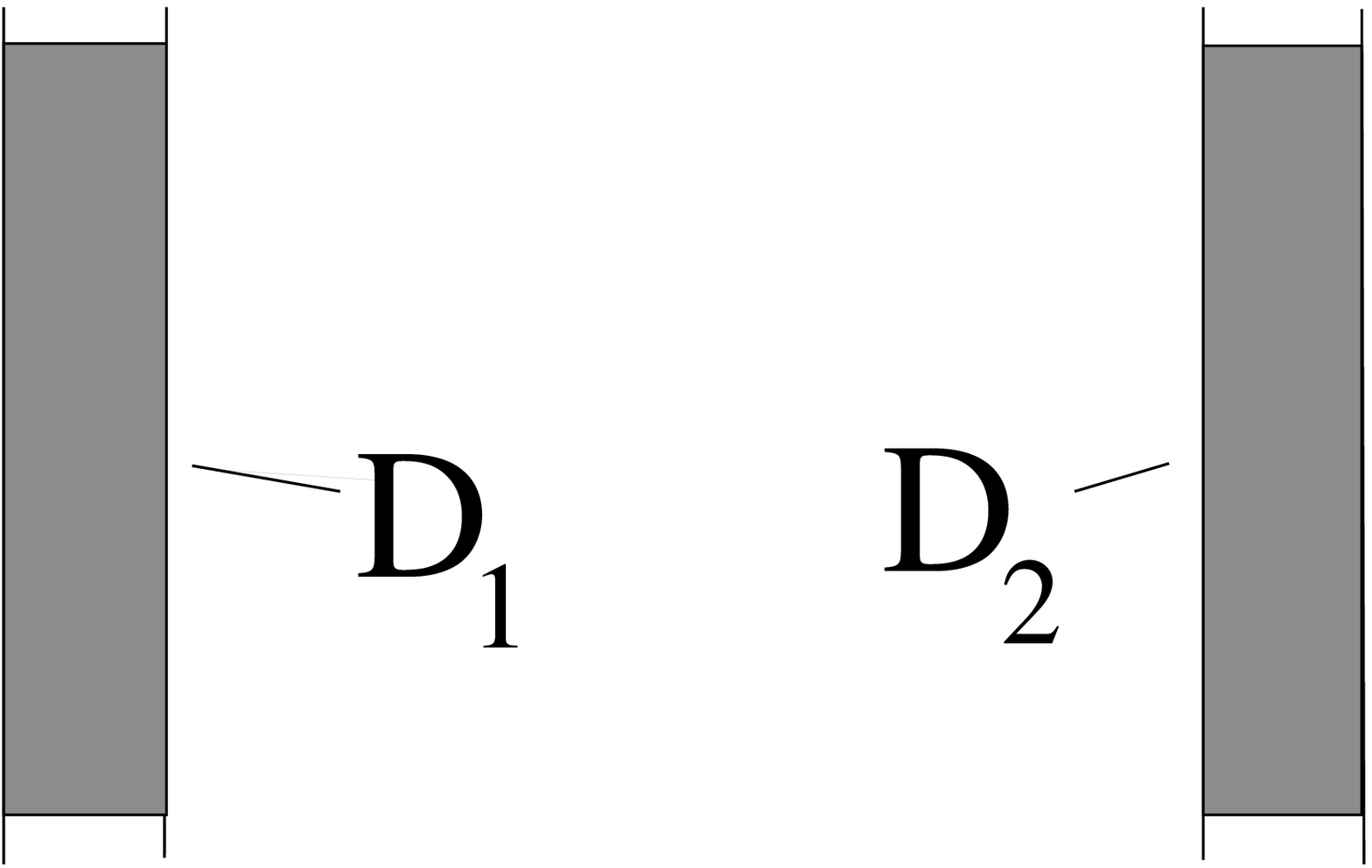}&\hspace{2cm}
\includegraphics[scale=0.3]{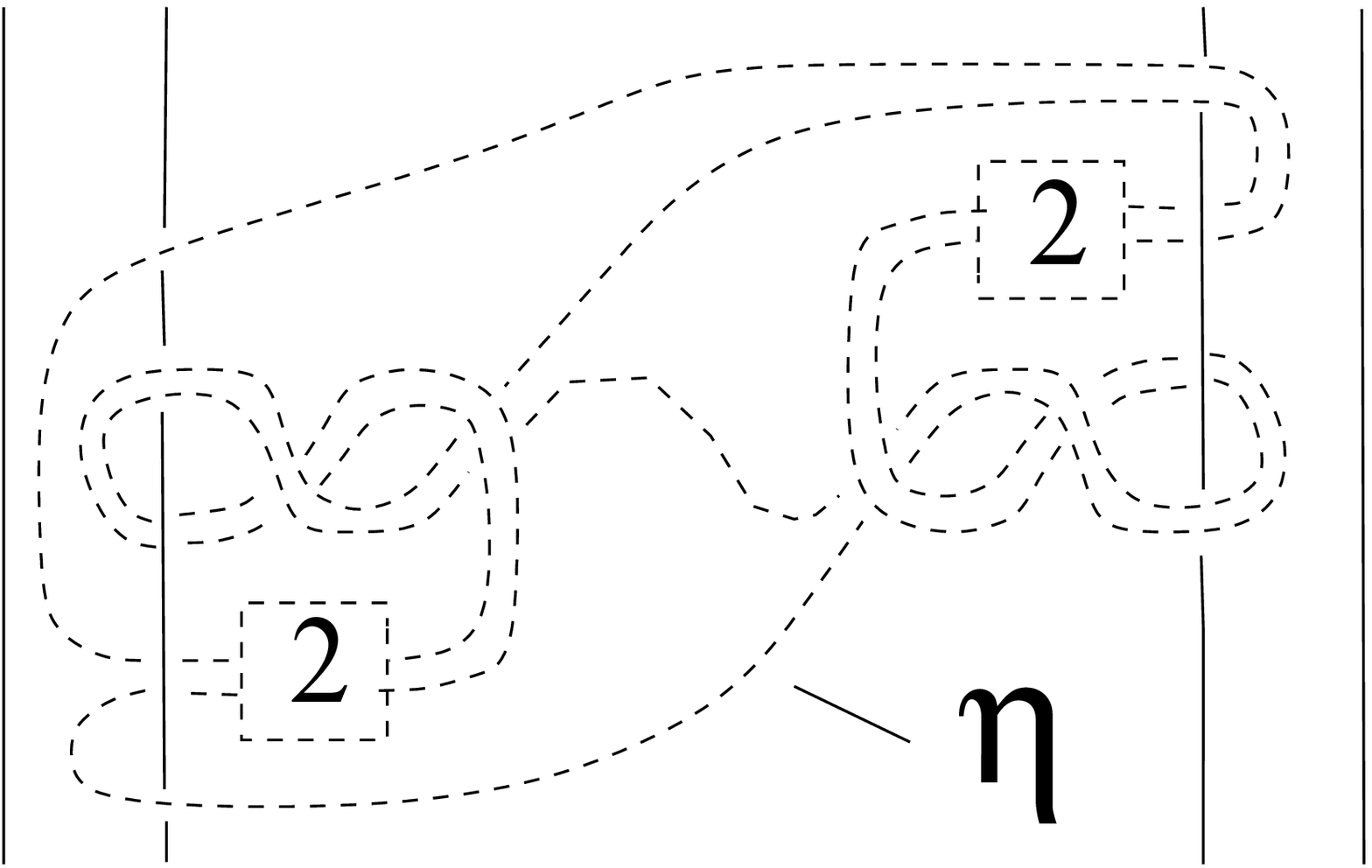}\\
(a)&(b)
\end{tabular}
\caption{Infection of a link by a knot.} \label{fig:satlink}
\end{center}
\end{figure}
Note that after resolving the self--intersections of $\eta$ we can contract $\eta$ in the
complement of $L$ to the trivial knot, in particular $\eta$ is homotopically trivial in $D^4\sm
D$. It now follows immediately from Corollary  \ref{corthm} that $I(L,\hJ,\eta)$ is topologically
slice for any knot $\hJ$. Also note that $\eta$ is unknotted in the complement of each component, in
particular if $L$ has components $L_1$ and $L_2$, then $I(L,\hJ,\eta)$ also has components $L_1$ and
$L_2$, albeit linked differently.

In the very special case that $L$ is the trivial link and $D_1$ and $D_2$ are disjointly embedded disks, then one can see that $I(L,\hJ,\eta)$ is the Bing double of the Whitehead double of $\hJ$. It is well--known that this link is in general not smoothly slice. We
refer to \cite{Ci06} to a summary of known obstructions to the Bing double of a Whitehead double
being smoothly slice.

\end{document}